\documentclass[onefignum,onetabnum]{siamart190516}

\usepackage{graphicx}
\graphicspath{{../}}
\usepackage{epstopdf}
\ifpdf
  \DeclareGraphicsExtensions{.eps,.pdf,.png,.jpg}
\else
  \DeclareGraphicsExtensions{.eps}
\fi

\usepackage{xr}

\makeatletter
\newcommand*{\addFileDependency}[1]{
  \typeout{(#1)}
  \@addtofilelist{#1}
  \IfFileExists{#1}{}{\typeout{No file #1.}}
}
\makeatother

\newsiamremark{remark}{Remark}
\newsiamremark{hypothesis}{Hypothesis}
\crefname{hypothesis}{Hypothesis}{Hypotheses}
\newsiamthm{claim}{Claim}

\newsiamthm{assumption}{Assumption}

\headers{ADFS}{H. Hendrikx, F. Bach and L. Massouli\'e}

\title{An Optimal Algorithm for Decentralized\\
Finite Sum Optimization
}

\author{Hadrien Hendrikx\thanks{INRIA - DIENS - PSL Research University 
  (\email{hadrien.hendrikx@inria.fr}, \email{francis.bach@inria.fr}, \email{laurent.massoulie@inria.fr})} 
\and Francis Bach\footnotemark[1]
\and Laurent Massouli\'e\footnotemark[1]} 

\usepackage{amsopn}

\usepackage{microtype}
\usepackage{subcaption}
\usepackage{booktabs}
\usepackage{dsfont}
\usepackage{xcolor}
\newcommand{\citep}{\cite}
\newcommand{\citet}{\cite}

\usepackage{algorithm}
\usepackage{algorithmic}

\usepackage[mathcal]{eucal}

\usepackage[utf8]{inputenc}

\usepackage{amssymb}
\usepackage{amsfonts}

\usepackage{amsmath}
\usepackage{dsfont}

\newcommand{\esp}[1]{\mathbb{E}\left[ #1 \right]}
\newcommand{\grando}[1]{O\left(#1 \right)}
\newcommand{\R}{\mathbb{R}}
\newcommand{\adfs}{ADFS}
\newcommand{\esdacd}{ESDACD}
\newcommand{\msda}{MSDA}
\newcommand{\ndim}{{\tilde{d}}}
\newcommand{\Ker}{{\rm Ker}}
\newcommand{\pcomm}{p_{\rm comm}}
\newcommand{\pcomp}{p_{\rm comp}}

\newcommand{\comm}{{\rm comm}}
\newcommand{\comp}{{\rm comp}}
\newcommand{\prox}{{\rm prox}}
\newcommand{\poso}{{{\rm pos}_0}}

\newcommand{\cM}{\mathcal{M}}
\newcommand{\cN}{\mathcal{N}}

\ifpdf
\hypersetup{
  pdftitle={An Optimal Algorithm for Decentralized Finite Sum Optimization},
  pdfauthor={H. Hendrikx, F. Bach and L. Massouli\'e}
}
\fi

\begin{document}

\maketitle

\begin{abstract}
Modern large-scale finite-sum optimization relies on two key aspects: distribution and stochastic updates. For smooth and strongly convex problems, existing decentralized algorithms are slower than modern accelerated variance-reduced stochastic algorithms when run on a single machine, and are therefore not efficient. Centralized algorithms are fast, but their scaling is limited by global aggregation steps that result in communication bottlenecks. In this work, we propose an efficient \textbf{A}ccelerated \textbf{D}ecentralized stochastic algorithm for \textbf{F}inite \textbf{S}ums named ADFS, which uses local stochastic proximal updates and decentralized communications between nodes. On $n$ machines, ADFS minimizes the objective function with $nm$ samples in the same time it takes optimal algorithms to optimize from $m$ samples on one machine. This scaling holds until a critical network size is reached, which depends on communication delays, on the number of samples $m$, and on the network topology. We give a lower bound of complexity to show that ADFS is optimal among decentralized algorithms. To derive ADFS, we first develop an extension of the accelerated proximal coordinate gradient algorithm to arbitrary sampling. Then, we apply this coordinate descent algorithm to a well-chosen dual problem based on an augmented graph approach, leading to the general ADFS algorithm. We illustrate the improvement of ADFS over state-of-the-art decentralized approaches with experiments.
\end{abstract}

\section{Introduction}

The success of machine learning models is mainly due to their capacity to train on huge amounts of data. Distributed systems can be used to process more data than one computer can store or to increase the pace at which models are trained by splitting the work among many computing nodes. In this work, we focus on problems of the form:\vspace{-3pt}
\begin{equation}
\label{eq:distributed_problem}
    \min_{\theta \in \mathbb{R}^d}\ \  \sum_{i=1}^n f_i(\theta),\ \ \ \ \mbox{ where } \ \ \ \  f_i(\theta) = \sum_{j=1}^m f_{i,j}(\theta) +  \frac{\sigma_i}{2}\| \theta \|^2 . \! \! \! \! \! \! \! \! \!
\end{equation}
This is the typical $\ell_2$-regularized empirical risk minimization problem with $n$ computing nodes that have $m$ local training examples each. The function $f_{i,j}$ represents the loss function for the $j$-th training example of node $i$ and is assumed to be convex and $L_{i,j}$-smooth~\citep{nesterov2013introductory, bubeck2015convex}. This kind of problems also arise in other applications, such as distributed resource allocation~\citep{xiao2006optimal} or distributed power control~\citep{molzahn2017survey}.

These problems are usually solved by first-order methods, and the basic distributed algorithms compute gradients in parallel over several machines~\citep{nedic2009distributed}. Another way to speed up training is to use \emph{stochastic} algorithms \citep{bottou2010large, defazio2014saga, johnson2013accelerating}, that take advantage of the finite sum structure of the problem to use cheaper iterations while preserving fast convergence. Lower bounds with matching optimal algorithms exist separately in both the finite-sum~\citep{lan2017optimal} and the distributed setting~\citep{scaman2017optimal}. This paper aims at bridging the gap between these two lines of work when local functions are smooth and strongly convex. In particular, we give lower complexity bounds for the distributed finite-sum setting, as well as ADFS, an algorithm that matches these bounds. Our contributions are the following, ordered by appearance in the paper:
\begin{enumerate}
    \item Tight lower complexity bounds. We recover as special cases the bounds from~\citet{scaman2018optimal} when $m=1$ (local functions are not finite sums), and the bounds from~\citet{lan2017optimal} when $n=1$ (there is only one machine).
    \item Generalization of the Accelerated Proximal Coordinate Gradient algorithm~\citep{lin2015accelerated, fercoq2015accelerated} to arbitrary sampling of blocks and strong convexity in a subspace.
    \item ADFS, a decentralized stochastic algorithm that matches our lower complexity bounds, and recovers the rate of MSDA~\citep{scaman2017optimal} when $m=1$ and the rate of optimal single machine stochastic algorithms~\citep{lin2015accelerated, defazio2016simple} when $n=1$. 
\end{enumerate}

The present paper is an extended journal version of the conference paper that introduces ADFS~\cite{hendrikx2019accelerated}. In particular, this paper presents a more flexible version of ADFS that can use synchronous rounds instead of local operations with global scheduling, and the contributions listed above were not present in the original work. We discuss the differences with the conference paper more in details later in the article. We now precisely define our setting, and discuss relevant related work. 

\section{Model and notations}
\subsection{Optimization problem}
 In the rest of this paper, following Scaman et al.~\citep{scaman2017optimal}, we assume that:
 
 \begin{itemize}
\item Each node (computing unit) $i \in \{1, ..., n\}$ can compute first-order characteristics, such as the gradient of its own functions $\nabla f_{i,j}$, the gradient of the Fenchel conjugate of its local function $\nabla f_i^*$, or the proximal operator of its own functions
\begin{equation}
     \label{eq:prox_def}
     \prox_{\eta f_{i,j}}(x) = \arg \min_v \frac{1}{2\eta}\|v - x\|^2 + f_{i,j}(v) \ \text{ for } x \in \R^d.
     \end{equation}
We assume that computing first-order characteristics for one function $f_{i,j}$ takes time $1$, and so that computing them for the function $f_i$ takes time $m$. This hides the fact that computing the proximal operator of a function is generally significantly more expensive than computing its gradient. Yet, this enables easier comparison between methods, and the difference between computing the proximal operator compared to the gradient of a single function is only of a constant factor in the case of generalized linear models such as least-squares or logistic regression.  
     \item Nodes are linked by a communication network and can only exchange messages (\emph{i.e.}, vectors in $\R^d$) with their neighbours. We assume that communications take time $\tau$. There are at this point no other restrictions on the communications, that can happen asynchronously and in parallel.
 \end{itemize}

Following notations from~\citet{xiao2017dscovr}, we define the batch condition number $\kappa_b$, which is a classical quantity in optimization, such that for all $i$, 
\begin{equation}
    \kappa_b \geq L_b(i) / \sigma_i \hbox{ where } \forall x, \ L_b(i) \geq \lambda_{\max}(\nabla^2 f_i(x)).
\end{equation}
Similarly, we define the stochastic condition number $\kappa_s$, which is a classical quantity in the analysis of finite sum opitmization problems, such that 
\begin{equation}
    \kappa_s \geq 1 + \frac{1}{\sigma_i}\sum_{j=1}^m L_{i,j} \hbox{ where } \forall x, \ L_{i,j} \geq \lambda_{\max}(\nabla^2 f_{ij}(x)).
\end{equation}
The iteration complexity of batch optimization methods such as gradient descent is proportional to $\kappa_b$, but they need to evaluate a full gradient (\emph{i.e.} $m$ individual gradients) at each step. On the other hand, the iteration complexity of stochastic variance-reduced algorithms~\citep{johnson2013accelerating, defazio2014saga, shalev2013stochastic} depends on $m + \kappa_s$, but only use one individual gradient at each iteration. Therefore, stochastic variance-reduced methods improve over batch methods by replacing their $O(m\kappa_b)$ time complexity by $O(m + \kappa_s)$. Yet, since $\nabla^2 f_i(x) = \sigma_i I_d + \nabla^2 f_{ij}(x) \leq (\sigma_i + \sum_{j=1}^m L_{i,j})I_d$ then we directly obtain that $L_b(i) \leq \sigma_i + \sum_{i=1}^m L_{i,j}$, and so $\kappa_b \leq \kappa_s$. Similarly, $\nabla^2 f_i(x) \geq \nabla^2 f_{ij}(x)$ and so $L_{i,j} \leq L_b(i)$ for all $j$. Therefore, we always have 
\begin{equation}\label{eq:kappa_b_kappa_s}
    (m + 1) \kappa_b \geq \kappa_s \geq \kappa_b.    
\end{equation}
This means that finite-sum methods gain nothing in the worst case, in which all $f_{i,j}$ are independent. However, the practical superiority of these methods suggests that $\kappa_s << m \kappa_b$ in many applications since samples are often correlated correlated. 

The upper bound on $\kappa_s$ is tight when $\max_x \lambda_{\max}(\nabla^2 f_i(x)) = \max_x \lambda_{\max}(\nabla^2 f_{ij}(x))$. Equality happens in particular in the extreme case in which all $\nabla^2 f_{ij}(x)$ are orthogonal, meaning that the sum is separable and optimization can be performed separately for each function. Considering least squares regression problems is also convenient to understand the difference between $\kappa_s$ and $\kappa_b$ more in details. In particular, if we denote by $C \in \R^{m \times m}$ the covariance matrix of the data, then $\kappa_b = \lambda_{\max}(C)$ the largest eigenvalue of $C$ and $\kappa_s = {\rm Tr}(C)$. In this case, it is clear that $\kappa_s << m \kappa_b$ unless the covariance matrix is close to isotropic. In summary, our goal is to replace the $m\kappa_b$ computational time factors by $m + \kappa_s$. Whether this improves the global time complexity depends on the structure of the problem (and thus of the data) but this is generally verified in practice. 

\subsection{Decentralized Communications}
The focus of this paper is on the \emph{decentralized} setting. In this case, \emph{gossip} algorithms~\citep{boyd2006randomized,nedic2009distributed,shi2015extra,nedic2017achieving} are generally used. Gossip communication steps consist in averaging gradients or parameters with neighbours, and can thus be abstracted as multiplication by a so-called gossip matrix $W$, which is an $n \times n$ symmetric positive semi-definite matrix such that ${\rm Ker}(W) = {\rm Span}(\mathds{1})$ where~$\mathds{1}$ is the constant vector of all ones, and ${\rm Span}(\mathds{1})$ denotes the vector space spanned by this vector.
Besides, $W$ is defined on the edges of the network, meaning that $W_{k\ell} = 0$ if $\ell \neq k$ and $\ell \notin \mathcal{N}(k)$, the set of the neighbours of node $k$.

A simple choice of gossip matrix is $L$, the Laplacian matrix of the graph, which is such that $L_{k \ell} = {\rm degree}(k)$ if $k = \ell$, $L_{k\ell} = 1$ if $k \in \mathcal{N}(\ell)$ and $L_{k \ell} = 0$ otherwise. We denote $\lambda_{\min}^+(W)$ the smallest non-zero eigenvalue of the matrix $W$, and the \emph{eigengap} of the gossip matrix (also called spectral gap) is defined as $\gamma = \lambda_{\min}^+(W) / \lambda_{\max}(W)$. This natural constant appears in the running time of many decentralized algorithms, and $\gamma^{-1/2}$ is often close to the diameter of the graph. For instance, $\gamma^{-1/2} = 1$ for the complete graph, $\gamma^{-1/2} = 2n / \pi$ for linear graphs and $\gamma^{-1/2} = O(\sqrt{n})$ for the 2D grid. More generally, $\gamma^{-1/2} \geq \frac{\Delta}{2 \sqrt{2} \ln_2(n)}$ for regular networks~\citep{alon1985lambda1}.

\section{Related work}

\begin{table*}[t]
\centering 
\begin{small}
\begin{sc}
\begin{tabular}{lcccc}
\toprule
Algorithm & Synchrony & Stochastic & Time \\
\midrule
Point-SAGA~\citep{defazio2016simple}  & N/A & \checkmark & $nm + \sqrt{n m\kappa_s}$\\
MSDA~\citep{scaman2017optimal} & Global & $\times$ & $\sqrt{\kappa_b}\left(m + \frac{\tau}{\sqrt{\gamma}}\right)$\\
ESDACD~\citep{hendrikx2018accelerated}  & Local & $\times$ & $\left(m + \tau \right)\sqrt{\frac{\kappa_b}{\gamma}}$\\
DSBA~\citep{shen2018towards}  & Global & \checkmark & $\left(m + \kappa_s + \gamma^{-1}\right) \left(1 + \tau\right)$\\
ADFS-Asynch~\citep{hendrikx2019accelerated} & Local & \checkmark & $m + \sqrt{m\kappa_s} + (1 + \tau)\sqrt{ \frac{\kappa_s}{\gamma}}$\\
ADFS-Synch (This paper) & Global & \checkmark & $m + \sqrt{m\kappa_s} + \tau \sqrt{ \frac{\kappa_\comm}{\gamma}}$\\
\bottomrule
\end{tabular}
\end{sc}
\end{small}
\caption{Comparison of various state-of-the-art decentralized algorithms to reach accuracy $\varepsilon$ in regular graphs. Constant factors are omitted, as well as the $\log\left(\varepsilon^{-1}\right)$ factor in the \emph{\textsc{Time}} column. The reported runtime for Point-SAGA corresponds to running it on a single machine with $n m$ samples. To allow for direct comparison, we assume that computing a dual gradient of a function $f_i$ as required by MSDA and ESDACD takes time $m$, although it is generally more expensive than to compute $m$ separate proximal operators of single $f_{i,j}$ functions. Rates reported are for a homogeneous setting, \emph{i.e.}, when all nodes have the same strong convexity parameter. For generalized linear models such as logistic regression, the $\kappa_\comm$ term in the rate of ADFS-Synch is defined in Lemma~\ref{lemma:rate_adfs} and is of order $\kappa_\comm = O(\kappa_b)$.} 
\label{fig:table_speeds}
\vskip -0.1in
\end{table*}
 The next paragraphs discuss the relevant state of the art for both distributed and stochastic methods, and Table~\ref{fig:table_speeds} sums up the speeds of the main decentralized algorithms available to solve Problem~\eqref{eq:distributed_problem}. Although it is not a distributed algorithm, Point-SAGA~\citep{defazio2016simple}, an optimal single-machine algorithm, is also presented for comparison.\vspace{3pt}

\paragraph{Centralized gradient methods}
A simple way to split work between nodes is to distribute gradient computations and to aggregate them on a parameter server. Provided the network is fast enough, this allows the system to learn from the datasets of~$n$ workers in the same time one worker would need to learn from its own dataset. Yet, these approaches are very sensitive to stochastic delays, slow nodes, and communication bottlenecks. Asynchronous methods may be used~\citep{recht2011hogwild, leblond2016asaga, xiao2017dscovr} to address the first two issues, but computing gradients on older (or even inconsistent) versions of the parameter harms convergence~\citep{chen2016revisiting}. Therefore, this paper focuses on decentralized algorithms, which are generally less sensitive to communication bottlenecks~\citep{lian2017can}.\vspace{3pt}

\paragraph{Decentralized gradient methods} In their synchronous versions, decentralized algorithms alternate rounds of computations (in which all nodes compute gradients with respect to their local data) and communications, in which nodes exchange information with their direct neighbors~\citep{duchi2012dual,shi2015extra, nedic2017achieving, tang2018d}. MSDA~\citep{scaman2017optimal} is a batch decentralized synchronous algorithm, and it is optimal with respect to the constants $\gamma$ and $\kappa_b$, among batch algorithms that can only perform these two operations. Instead of performing global synchronous updates, some approaches inspired from gossip algorithms~\citep{boyd2006randomized} use randomized pairwise communications~\citep{nedic2009distributed, johansson2009randomized, colin2016gossip}. This for example allows fast nodes to perform more updates in order to benefit from their increased computing power. These randomized algorithms do not suffer from the usual worst-case analyses of bounded-delay asynchronous algorithms, and can thus have fast rates because the step-size does not need to be reduced in the presence of delays. For example, ESDACD~\citep{hendrikx2018accelerated} achieves the same optimal speed as MSDA when batch computations are faster than communications ($\tau > m$). However, both algorithms are obtained using a dual approach~\citep{uribe2020dual}. Therefore, they require gradients of the Fenchel conjugates of the full local functions, which are generally much harder to get than regular gradients.\vspace{3pt}

\paragraph{Stochastic algorithms for finite sums}
All distributed methods presented earlier are \emph{batch} methods that rely on computing \emph{full gradient} steps of each function $f_i$. Stochastic methods perform updates based on randomly chosen functions $f_{i,j}$. In the smooth and strongly convex setting, they can be coupled with \emph{variance reduction}~\citep{schmidt2017minimizing,shalev2013stochastic, johnson2013accelerating, defazio2014saga} and \emph{acceleration}, to achieve the $m + \sqrt{m\kappa_s}$ optimal finite-sum rate, which significantly improves over the $m\sqrt{\kappa_b}$ batch optimum when the dataset is large. Examples of such methods include Accelerated-SDCA~\citep{shalev2014accelerated}, APCG~\citep{lin2015accelerated}, Point-SAGA~\citep{defazio2016simple} or Katyusha~\citep{allen2017katyusha}.\vspace{3pt}

\paragraph{Decentralized stochastic methods} In the smooth and strongly convex setting, DSA~\citep{mokhtari2016dsa} and later DSBA~\citep{shen2018towards} are two linearly converging stochastic decentralized algorithms. DSBA uses the proximal operator of individual functions $f_{i,j}$ to significantly improve over DSA in terms of rates. Yet, DSBA does not enjoy the $\sqrt{m\kappa_s}$ accelerated rate, and needs an excellent network with very fast communications. Indeed, nodes need to communicate each time they process a single sample, resulting in many communication steps. Other approaches based on SGD exist~\citep{koloskova2019decentralized}, but they do not use variance reduction and thus do not converge linearly. Therefore, to the best of our knowledge, there is no decentralized stochastic algorithm with accelerated linear convergence rate or low communication complexity without sparsity assumptions (\emph{i.e.}, sparse features in linear supervised learning).\vspace{3pt}

\paragraph{ADFS} The main contribution of this paper is a locally synchronous \textbf{A}ccelerated \textbf{D}ecentralized stochastic algorithm for \textbf{F}inite \textbf{S}ums, named ADFS. It reduces to APCG for empirical risk minimization~\citep{lin2015accelerated} in the limit case $n=1$ (single machine), and therefore then has a $m + \sqrt{m\kappa_s}$ convergence rate. Besides, this rate stays unchanged when the number of machines grows, meaning that ADFS can process $n$ times more data than APCG in the same amount of time on a network of size $n$. This scaling lasts as long as $\tau\sqrt{\kappa_\comm} \gamma^{-\frac{1}{2}} < m + \sqrt{m \kappa_s}$, meaning that the number of nodes can be arbitrarily large as long as delays are small enough. Therefore, \adfs~outperforms both \msda~and DSBA, combining optimal network scaling with the efficient distribution of optimal sequential finite-sum algorithms. Note however that, similarly to DSBA and Point-SAGA, ADFS requires evaluating ${\rm prox}_{f_{i,j}}$, which requires solving a local optimization problem. Yet, in the case of linear models such as logistic regression, it is only a constant factor slower than computing $\nabla f_{i,j}$, and it is especially much faster than computing the gradient of the conjugate of the full dual functions $\nabla f_i^*$ required by ESDACD and MSDA.
\vspace{3pt}

\paragraph{Improvements over the conference paper} This paper is based on the ADFS conference paper~\citep{hendrikx2019accelerated}. Yet, it is not a strict extension, and some parts have been removed in order to ease the reading and focus on contributions more related to optimization. In particular, the locally synchronous aspect of ADFS has been dropped in favor of standard synchronous gossip, which allows to remove the sections about time and scheduling. This paper is based on arguments that are similar to the ones used in the conference paper, but it presents new and stronger results.
First of all, we introduce a lower bound that was not present in the conference paper. Then, we extend the accelerated proximal coordinate descent algorithm, which is the algorithmic core of ADFS, to work with blocks of coordinates. This allows to present a synchronous version of ADFS, which is both simpler and faster when communication and computation delays are homogeneous. Furthermore, we introduce the constant $\kappa_\comm$, which captures the impact of the relationship between the topology of the graph and the regularity of local functions on the iteration complexity of ADFS. This allows us to obtain tight results on the communication complexity of ADFS and show that ADFS is actually optimal since it matches the lower bound. Note that the locally version of ADFS from the conference paper did not enjoy optimal runtime because of scheduling issues and a looser analysis. Therefore, and although they build on the same ideas as the conference paper, all results presented in this paper are novel and contribute to building a much more consistent theory.\vspace{3pt}

The first contribution of this paper is a lower bound for distributed finite sum optimization, presented in Section~\ref{sec:lower_bound}. Then, we introduce in Section~\ref{sec:apcg_main} our second contribution, a generalization of APCG that works with arbitrary sampling of blocks of coordinates. Our last contribution is ADFS, obtained by applying the previous APCG algorithm to a novel augmented graph approach formulation presented in Section~\ref{sec:model}. The generic ADFS algorithm is presented in Section~\ref{sec:alg}. Finally, Section~\ref{sec:perfs} presents a relevant choice of parameters leading to the rates shown in  Table~\ref{fig:table_speeds}, and an experimental comparison is done in Section~\ref{sec:experiments}. A Python implementation of ADFS is also provided in supplementary material.

\section{Optimal rates}
\label{sec:lower_bound}
Many of the algorithms discussed in the previous sections are proven to be optimal in specific settings. In particular, APCG (when applied to the dual of empirical risk minimization problems) and Point-SAGA are proven to be optimal among single-machine algorithms to solve finite-sum problems~\citep{lan2017optimal}. Similarly, MSDA is optimal among batch decentralized algorithms~\citep{scaman2017optimal}. Although other optimality results have recently been proven when removing the strong convexity and smoothness assumptions in the distributed setting~\citep{scaman2018optimal}, there is, to the best of our knowledge, no lower bound for distributed optimization when local functions are themselves finite sums. We fill this gap in this section by extending the decentralized lower bound of~\citet{scaman2017optimal} to the finite sum setting, using worst-case functions inspired from the single-machine finite-sum lower bound~\citet{lan2017optimal}.

\subsection{Black Box Model}
\label{sec:black_box_model}
The notion of black-box optimization procedure that we use is largely based on~\citet{scaman2018optimal}. The main difference is that nodes have many local functions but they only choose one (possibly at random) at each step to perform their update. More specifically, we consider distributed algorithms that respect:
\begin{enumerate}
  \item \textbf{Local memory:} each node $i$ can store past values in an internal memory $\cM_{i,t} \subset \R^d$ at time $t \geq 0$. The values in this local memory can come either from local computation or communication, so that for all $i \in \{1, \cdots n\}$, $\cM_{i,t} \subset \cM_{i,t}^\comm \cup \cM_{i,t}^\comp$.

  \item \textbf{Local computation:} each node can, at time $t$, compute $\nabla f_{i, \zeta_t}(\theta)$, $\nabla f_{i, \zeta_t}^*(\theta)$ and $\prox_{\eta f_{i, \zeta_t}}(\theta)$ for some $\eta > 0$, where $\zeta_t \in \{1, \cdots, m\}$ is fixed for a given $t$ (but may be chosen by the algorithm). This means that
  $$ \cM_{i,t}^\comp = {\rm Span}\left(\big\{\theta, \nabla f_{i, \zeta_t}(\theta), \nabla f_{i, \zeta_t}^*(\theta), \prox_{\eta f_{i, \zeta_t}}(\theta): \theta \in \cM_{i, t-1}\big\},  \eta \geq 0 \right).$$

  \item \textbf{Local communication:} each node can, at time $t$, share a value to its neighbours so that for all $i \in \{1, \cdots, n\}$,
$$\cM_{i,t}^\comm = {\rm Span}\left(\cup_{j \in \cN(i)} \cM_{j, t-\tau}\right).$$

  \item \textbf{Output value:} each node $i$ must, at time  specify one vector in its memory as local output of the algorithm, that is, for all $i \in \{1, \cdots, n\}$, $\theta_{i,t} \in \cM_{i,t}$.
\end{enumerate}

\vspace{5pt}The main difference with the definition from~\citet{scaman2018optimal} is that at each step, the first order characteristics are only computed for one summand (the one with index $\zeta_t$) of the local finite sum of node $i$.

\subsection{Lower bounds}
\label{sec:lower_bound_subsec}
We first present a general lower bound for the distributed optimization setting. More specifically, we show that for any black-box optimization procedure, at least $\Omega((m + \sqrt{m\kappa_s})\log(1 / \varepsilon))$ computation steps and $\Omega(\tau \sqrt{\kappa_\ell / \gamma} \log(1 / \varepsilon))$ communication steps are needed. This lower bound is not surprising since it is similar to that of~\citet{scaman2017optimal}, but the lower bound on the computation cost is replaced by the standard finite-sum lower-bound for the computation cost~\citep{lan2017optimal}. Theorem~\ref{thm:lb_general} shows that the lower bound for both communications and computations can be achieved by the same function. Lower bound proofs for first-order methods usually rely on the fact that in the work case, the algorithms can make progress in at most one dimension per oracle call~\citep{nesterov2013introductory}. This means that the lower bounds are valid only for a number of iterations~$t$ that depends on the dimension of the problem. In order to avoid this dependency, we prove a result in $\ell_2$, the space of square summable sequences. Yet, a similar result with a similar proof would hold in $\R^d$.

\begin{theorem}
\label{thm:lb_general}
Let $\mathcal{G}$ be a graph of size $n > 0$ and diameter $\Delta$, and $\kappa_\ell > 0$. There exist $n \times m$ functions $f_{i,j}: \ell_2 \rightarrow \R$ such that each $f_{i,j}$ is convex and $L_{i,j}$-smooth, $f_i \hat{=}\sum_{j=1}^m f_{i,j}$ is $L_i$-smooth and $\sigma_i$-strongly convex with $L_{i,j}$ and $\sigma_i$ such that $\kappa_\ell \geq L_i / \sigma_i \geq L_{i,j} / \sigma_{i}$ for all $i,j$, and such that if $f_i$ is the local function of node $i$ then for any $t \geq 0$ and black-box procedure that generates and output $\theta^t$ such that $(\theta^t)_i \in \ell_2$ is the output value of node $i$ at time $t$, one has:
\begin{align*}
    2 \frac{1 - q^2}{1 - q} \esp{\frac{\|\theta^t - \theta^*\|^2}{\|\theta^0 - \theta^*\|^2}} \geq
  \left(1 - \frac{2m}{m + \sqrt{m\kappa_s / 3})}\right)^{\frac{4\lceil t \rceil}{m}} + \left(1 - \frac{2}{1 + \sqrt{\kappa_\ell / 3}}\right)^{2 + \frac{2\lceil t\rceil}{\Delta\tau}}.
\end{align*}
where $\kappa_s \geq \sum_{j=1}^m L_{i,j} / \sigma_i$, $q = \frac{\sqrt{\kappa_\ell / 3} - 1}{\sqrt{\kappa_\ell / 3} + 1}$, and $\theta^* = \arg \min_\theta \sum_{i=1}^n f_i(\theta)$.
\end{theorem}

\begin{proof}
The proof relies on choosing particular functions that are hard to optimize locally and that require communication. Hard functions $f_i$ are chosen similar to that of~\citet{scaman2017optimal}, so that only a small set a of nodes can actually make progress towards the optimum at a given point in time, meaning that parallelism is very restricted. Then, $f_{i,j}$ are chosen such that $f_{i,j}(x) = f_i(e_j^\top x)$ for $x \in (\ell_2)^m$, so that progress along one $j \in \{1, \cdots, m\}$ does not result in progress along the other dimensions, as in~\citet{lan2017optimal}. The result is stated with $\ell_2$ instead of $(\ell_2)^m$ because if $m$ is finite and $x \in (\ell_2)^m$, then if $\theta$ is such that $\theta_{km + i} = (x_i)_k$ for $i \in \{1, \cdots, m\}$ and $k \in \mathbb{N}$ then $\theta \in \ell_2$.

Two extreme regimes are then considered, \emph{i.e.}, when communications are instant ($\tau = 0$) and when computations are instant ($\tau = \infty$). In the first case, very few nodes make progress at a given point in time so there is almost no parallelism and the complexity is the same as that of one node optimizing its own function. In the second case, the stochastic gradient aspect does not matter and the time taken by the algorithm is lower bounded by the time required for the information to go back and forth between the nodes that can actually make progress. The complete proof can be found in Appendix~\ref{app:lower_bound}.
\end{proof}

This bound can be further simplified into the asymptotic expression below:
\begin{corollary}[Centralized lower bound] \label{corr:centralized_lb}
Under the assumptions of Theorem~\ref{thm:lb_general}, there exist functions such that for any black-box procedure, the time to reach a precision $\varepsilon$ is lower bounded by: 
$$\Omega\left(\left[m + \sqrt{m\kappa_s} + \tau \Delta \sqrt{\kappa_\ell}\right]\log(\varepsilon^{-1})\right).$$
\end{corollary}

The previous lower bounds rely on the diameter of the network, without assuming any structure. We use in this section the same trick as in~\cite{scaman2017optimal} to extend the lower bounds to the gossip communications setting. 

\begin{corollary}[Decentralized lower bound]
\label{corr:decentralized_lb}
Let $\gamma > 0$, and $\kappa_\ell > 0$. There exist a gossip matrix $W$ with spectral gap $\gamma$ and $n \times m$ functions $f_{i,j}: \ell_2 \rightarrow \R$ such that each $f_{i,j}$ is convex and $L_{i,j}$-smooth, $f_i \hat{=}\sum_{j=1}^m f_{i,j}$ is $\sigma_i$-strongly convex with $L_{i,j}$ and $\sigma_i$ such that $\kappa_\ell \geq L_{i,j} / \sigma_{i}$ for all $i,j$, and such that if $f_i$ is the local function of node $i$ then for any black-box procedure, the time to reach precision $\varepsilon$ is lower bounded by:
$$\Omega\left(\left[m + \sqrt{m\kappa_s} + \tau \sqrt{\frac{\kappa_\ell}{\gamma}}\right]\log(\varepsilon^{-1})\right).$$
\end{corollary}

\begin{proof}
The proof relies on the fact that for all $\gamma > 0$, it is possible to construct a gossip matrix on a line graph of size $n$ with spectral gap $\gamma = 0$. In this case, the diameter of the graph is $n$, which is of order $\gamma^{-1/2}$. Details can be found in Theorem~2~\cite{scaman2017optimal}.
\end{proof}

It is interesting to remark that considering the finite-sum setting only changes the lower bound on the computation cost. This is not surprising since it only allows to compute cheaper stochastic gradients but cannot reduce communication cost without additional assumptions on the functions used. As a matter of fact, the computation and communication aspects are treated separately in the lower bound. This could suggest room for improvement for this lower bound. Yet, the bound we obtain is actually tight since it is matched by the ADFS-Synch algorithm. There is actually a small subtle gap between the lower and the upper bounds, which is caused by the fact that the communication lower bound depends on $\kappa_\ell$, whereas the complexity of ADFS-Synch depends on $\kappa_\comm$, which can be much bigger. Yet, $\kappa_\comm = O(\kappa_\ell)$ in the case of the worst case function used for the lower bound, as shwon in Appendix~\ref{app:adfs_line_graph}. More generally $\kappa_\comm = O(\kappa_b)$ for generalized linear models such as linear regression when the regularization parameter is the same for all nodes, which is a prime use-case for ADFS.

\subsection{Replicated dataset}
Assume that the $\sqrt{m\kappa_s}$ term dominates. In this case, optimal single-machine algorithms require $O(\sqrt{nm\kappa_s})$ iterations and so Theorem~\ref{thm:lb_general} suggests that the maximum speedup obtainable by any distributed algorithm in this setting is of $\sqrt{n}$. This result is surprising and seems to contradict the linear speedup obtained by Katyusha~\citep{allen2017katyusha}. This is because the speedup of Katyusha is based on mini-batching, which relies on the fact that all nodes sample the same functions. The lower bound proofs critically rely on choosing different functions for different nodes. In the setting of Theorem~\ref{thm:lb_general}, the size of the problem grows with the number of nodes. On the other hand, the linear speedup of Katyusha considers a problem with a fixed number of samples processed by an increasing number of nodes. In particular, the bound of Theorem~\ref{thm:lb_general} can be weakened to match the Katyusha complexity results when all nodes are forced to have the same local functions. The idea behind these results is that only one or two nodes actually contribute to reducing the error in the worst case (Theorem~\ref{thm:lb_general}), whereas this cannot happen if all nodes have the same local function. Note that the time aspect is overseen in Katyusha, and the network is simply expected to be ``fast enough". In the replicated setting, increasing $\tau$ only increases the runtime of Katyusha up to a certain point because nodes do not actually need to communicate to reach the optimum since they all have the same local functions. On the other hand, the theoretical rate of ADFS does not show improvements in the replicated setting. 

\section{Block Accelerated Proximal Coordinate Gradient with Arbitrary Sampling}
\label{sec:apcg_main}
Before we start with the actual distributed algorithm, we first introduce a coordinate descent method. Indeed, this is the main tool that we apply to a well-chosen dual formulation to derive ADFS. The convergence results of ADFS are based on the convergence of this Accelerated Proximal Coordinate Gradient method. ADFS is derived in a way that is similar to that of the classical  APCG algorithm~\citep{lin2015accelerated}, but we integrate the decentralized aspect, which requires several improvements over the original APCG. 

\subsection{General formulation}
\label{app:generalized_apcg}
In this section, we study the generic problem of accelerated proximal coordinate descent. We give an algorithm that works with \emph{arbitrary sampling} of \emph{blocks} of coordinates of arbitrary size, thus yielding a stronger result than state-of-the-art approaches~\citep{fercoq2015accelerated, lin2015accelerated}. This is a key contribution that allows to obtain fast rates when sampling probabilities are heterogeneous and determined by the problem. In the dual formulation of the problem, there is one coordinate per point in the dataset as well as one for each edge of the network. Therefore, the block aspect allows to have a synchronous algorithm by picking only coordinates of a given kind (data point or network edge) to perform computation and communication rounds. Similarly, arbitrary sampling is useful to pick different probabilities for computing and for communicating. To avoid any confusion with the rest of the paper, we note $\ndim$ the dimension of the problem that we wish to solve. More specifically, we study the following generic problem:
\begin{equation}
\label{eq:generic_problem}
    \min_{x \in \mathbb{R}^{\ndim}}\ \  q_A(x) + \sum_{i=1}^{\ndim} \psi_i(x^{(i)}),
\end{equation}
where all the functions $\psi_i$ are convex and $q_A$ is such that there exists a matrix $A$ such that $q_A$ is $(\sigma_A)$-strongly convex on ${\rm Ker}(A)^\perp$, the orthogonal of the kernel of $A$, as defined by Equation~\eqref{eq:sc_semi_norm}. For the problems that we will consider, ${\rm Ker}(A)^\perp \subsetneq \R^\ndim$ and so $q_A$ is not strongly convex on the whole space. We introduce matrix $A$ in order to recover the good properties ensured by strong convexity, with the difference that they now hold only on a subspace. We note $A^\dagger$ is the pseudo-inverse of $A$, meaning that $A^\dagger A$ is the projector on ${\rm Ker}(A)^\perp$. We sometimes abuse notations by writing $A
^{-\frac{1}{2}}$ instead of $(A^\dagger)^\frac{1}{2}$. The strong convexity on $\Ker(A)^\perp$ can be written as the fact that for all $x,y \in \mathbb{R}^\ndim$:
\begin{equation}
\label{eq:sc_semi_norm}
q_A(x) - q_A(y)   \geq   \nabla q_A(y)^\top \! A^\dagger A(x - y) + \textstyle \frac{\sigma_A}{2}(x - y)^\top \! A^\dagger A (x - y).    
\end{equation}
Note that this implies that $q_A$ is constant on ${\rm Ker}(A)$, so in particular there exists a function $q$ such that for any $x \in \mathbb{R}^\ndim$, $q_A(x) = q(Ax)$. In this case, $\sigma_A$ is such that $x^\top A^\top\nabla^2q(y) Ax \geq \sigma_A\|x\|^2$ for any $x\in {\rm Ker}(A)^\perp$ and $y\in \mathbb{R}^\ndim$. Besides, $q_A$ is assumed to be $(M)$-smooth on $\Ker(A)^\perp$, meaning that there exists a matrix $M$ such that: 
\begin{equation}
\label{eq:M_smoothness}
q_A(x) - q_A(y)   \leq  \nabla q_A(y)^\top \! A^\dagger A(x - y) + \textstyle \frac{1}{2}(x - y)^\top \! M (x - y).    
\end{equation}
The block-version of APCG with arbitrary sampling is presented in Algorithm~\ref{algo:generalized_apcg}, and we explicit its rate in Theorem~\ref{thm:gen_apcg}. 

\subsection{Algorithm and results}
In this section, we denote $e_i \in \mathbb{R}^\ndim$ the unit vector corresponding to coordinate $i$, and $x^{(i)} = e_i^\top x$ for any $x \in \R^\ndim$. Let $R_i = e_i^\top A^\dagger A e_i$ and $p_i$ be the probability that coordinate~$i$ is picked to be updated. For a batch of coordinates $b \subset \{1, \cdots, \ndim\}$, we introduce the random matrix $P_b$ which is a \emph{diagonal} matrix such that $(P_b)_{ii} = p_i$ if $i \in b$ and $(P_b)_{ii} = 0$ otherwise, where $p_i = \sum_{b, \ i \in b} p_b$ if $p_b$ is the probability of sampling block $b$. In particular, $\mathbb{E}\left[ P_b^\dagger \right] = Id$, where $P_b^\dagger$ is the pseudo-inverse of $P_b$. The matrix $P_b$ defines the sampling that is performed. This allows to have a flexible sampling with blocks of arbitrary sizes sampled with arbitrary probabilities. Constant $S$ is such that $S^2 \geq \lambda_{\max}(A^\dagger A P_b^\dagger M P_b^\dagger A^\dagger A)$ for all batches $b$, where we recall that $M$ is the smoothness of function $q_A$, as defined in Equation~\eqref{eq:M_smoothness}. Then, following the approach of Nesterov and Stich~\citep{nesterov2017efficiency}, we fix $A_0, B_0 \in \mathbb{R}$ and recursively define the sequences $\alpha_t, \beta_t, a_t, A_t$ and $B_t$ such that:
\begin{align*}
    & a_{t+1}^2 S^2 = A_{t+1} B_{t+1}, & B_{t+1} = B_t + \sigma_A a_{t+1}, & \ \ \ \ \ \ \ \ \ A_{t+1} = A_t + a_{t+1}, \\
    & \alpha_t = \frac{a_{t+1}}{A_{t+1}}, & \beta_t = \frac{\sigma_A a_{t+1}}{B_{t+1}}. & 
\end{align*} 
Finally, we introduce the sequences $(y_t)$, $(v_t)$ and $(x_t)$, that are all initialized at $0$, and $(w_t)$ such that for all $t$, $w_t = (1 - \beta_t) v_t + \beta_t y_t$. We define $\eta_{t} = \frac{a_{t+1}}{B_{t+1}}$ and the proximal operator ${\rm prox}_{\eta f}$ is defined in Equation~\eqref{eq:prox_def}.

\begin{algorithm}
\caption{Generalized APCG$(A_0, B_0, S, \sigma_A)$}
\label{algo:generalized_apcg}
\begin{algorithmic}
\STATE $y_0 = 0$, $v_0 = 0$, $t = 0$
\WHILE{$t < T$}
\STATE $y_t = \frac{(1 - \alpha_t) x_t + \alpha_t(1 - \beta_t)v_t}{1 - \alpha_t \beta_t}$
\STATE Sample $b_t$ with probability $p_{b_t}$
\STATE $v_{t+1} = v_{t+\frac{1}{2}} = (1 - \beta_t) v_t + \beta_t y_t - \eta_{t}P_b^\dagger \nabla q_A(y_t)$
\STATE $v_{t+1}^{(i)} = {\rm prox}_{\eta_t p_i^{-1} \psi_i}\left(v_{t+\frac{1}{2}}^{(i)}\right)$ for all $i \in b$
\STATE $x_{t+1} = y_t + \alpha_t P_b^\dagger A^\dagger A (v_{t+1} - (1 - \beta_t) v_t - \beta_t y_t)$
\ENDWHILE
\end{algorithmic}
\end{algorithm}
For generalized APCG to work well, the proximal operator needs to be taken in the subspace defined by the projector $A^\dagger A$, and so the non-smooth $\psi_i$ terms have to be separable after composition with $A^\dagger A$. Since $A^\dagger A$ is a projector, this constraint is equivalent to stating that either $R_i = 1$ (projection does not affect the coordinate $i$), or $\psi_i = 0$ (no proximal update to make).

\begin{assumption}
\label{assumption:gen_apcg}
The functions $q_A$ and $\psi$ are such that Equation~\eqref{eq:sc_semi_norm} holds for some $\sigma_A \geq 0$ and Equation~\eqref{eq:M_smoothness} holds for some $M$. Besides, $\psi$ and $A$ are such that either $R_i = 1$ or $\psi_i = 0$ for all $i \in \{1, ..., \ndim\}$.
\end{assumption}
This natural assumption allows us to formulate the proximal update in standard squared norm since the proximal operator is only used for coordinates $i$ for which $A^\dagger A e_i = e_i$. Then, we formulate Algorithm~\ref{algo:generalized_apcg} and analyze its rate in Theorem~\ref{thm:gen_apcg}.
\begin{theorem}
\label{thm:gen_apcg}
Let $F: x \mapsto q_A(x) + \sum_{i=1}^{\ndim} \psi_i\left(x^{(i)}\right)$ such that Assumption~\ref{assumption:gen_apcg} holds. If $S$ is such that $S^2 \geq \lambda_{\max}((A^\dagger A P_b^\dagger M P_b^\dagger A^\dagger A)$ for all $b$ and $1 - \beta_t - \frac{\alpha_t}{p_i} \geq 0$ for all $i$ such that $\psi_i \neq 0$, the sequences $v_t$ and $x_t$ generated by APCG verify:
\begin{equation*}
    B_t\esp{\|v_t - \theta^\star\|^2_{A^\dagger A}} + 2 A_t \left[\esp{F(x_t)} - F(\theta^\star)\right] \leq C_0,
\end{equation*}
where $C_0 = B_0 \|v_0 - \theta^\star\|^2 + 2A_0\left[F(x_0) - F(\theta^\star)\right]$ and $\theta^\star$ is a minimizer of $F$. The rate of APCG depends on $S$ through the sequences $\alpha_t$ and $\beta_t$. 
\end{theorem}

\begin{proof}[Sketch of proof]
The proof is an adaptation of the proofs from~\citet{lin2015accelerated} and~\citet{nesterov2017efficiency}. In particular, the structure is similar to that of~\citet{nesterov2017efficiency}. The difference is that the $\|v_{t+1} - \theta^\star\|^2$ is studied in norm $A^\dagger A$ and that $v_{t+1}$ cannot be expressed simply as $v_t$ minus a gradient term the way it was before because of the proximal update. Therefore, we develop $\|v_{t+1} - \theta^\star\|^2_{A^\dagger A}$ using the strong convexity of the proximal mapping instead, which is a key argument from~\citet{lin2015accelerated}.

The other key point of the APCG proof is that $x_t$ can be expressed as a convex combination of all the $v_l$ for $l \leq t$. This does not directly extend to the arbitrary sampling case because the coefficients may not be the same for all coordinates, so we need to prove that the convex combination property holds separately for each coordinate. This is possible because the only terms that depend on the coordinates in the decomposition of $x_t$ come from the $v_{t+1} - w_t$ term. Yet, $v_{t+1} = w_t$ when the coordinate is not picked, so we can still write that $x_{t+1}^{(i)} = y_t^{(i)} + \frac{\alpha_t}{p_i}(v_{t+1}^{(i)} - w_t^{(i)})$ even when coordinate $i$ is not picked at time $t$.
\end{proof}

\subsection{Explicit rates}

Theorem~\ref{thm:gen_apcg} is a general method that in particular requires to set values for $A_0$, $B_0$, $\alpha_0$ and $\beta_0$. The two following corollaries give choices of parameters depending on whether $\sigma_A > 0$ or $\sigma_A = 0$, along with the rate of APCG in these cases.

\begin{corollary}[Strongly Convex case]
\label{corr:sc_apcg}
Let $F$ be such that it verifies the assumptions of Theorem~\ref{thm:gen_apcg}. If $\sigma_A > 0$, we can choose for all $t\in \mathbb{N}$ $\alpha_t = \beta_t = \rho$ and $A_t = \sigma_A^{-1} B_t = (1 - \rho)^{-t}$ with $\rho = \sqrt{\sigma_A}S^{-1}$. In this case, the condition $1 - \beta_t - \frac{\alpha_t}{p_i} \geq 0$ can be weakened to $1 - \frac{\alpha_t}{p_i} \geq 0$ and it is automatically satisfied by our choice of $S$, $\alpha_t$ and $\beta_t$. In this case, the sequences $x_t$ and $v_t$ verify: $$\sigma_A \esp{\|v_t - \theta^\star\|^2_{A^\dagger A}} + 2\left[\esp{F(x_t)} - F(\theta^\star)\right] \leq C_0 (1- \rho)^t,$$
where $C_0 = \sigma_A \|v_0 - \theta^\star\|^2 + 2\left[F(x_0) - F(\theta^\star)\right]$.
\end{corollary}
Corollary~\ref{corr:sc_apcg} is the extension of the results of~\citet{lin2015accelerated} to block coordinates and arbitrary sampling. In particular, APCG converges linearly in this case, and we recover the rate of~\cite{lin2015accelerated} in the special case in which we choose blocks of size $1$ uniformly at random.  Note that an arbitrary sampling extension of accelerated coordinate descent was already present in~\cite{hanzely2018accelerated} but without the block or proximal aspects on which our technical contributions are focused. 

\begin{corollary}[Convex case]
\label{corr:cvx_apcg}
Let $F$ be such that it verifies the assumptions of Theorem~\ref{thm:gen_apcg}. If $\sigma_A = 0$, we can choose $\beta_t = 0$ and $\alpha_0 = p_{\min}^2$ with $p_{\min} = \min_{i: \psi_i \neq 0} p_i$. In this case, the condition $1 - \beta_t - \frac{\alpha_t}{p_i} \geq 0$ is always satisfied for our choice of $S$ and the error verifies:
$$ \esp{F(x_t)} - F(\theta^\star) \leq \frac{2}{t^2} \left[S^2 r_t^2 + \frac{2}{p_{\min}^2}\left[F(x_0) - F(\theta^\star)\right]\right],$$
with $r_t^2 = \|v_0 - \theta^\star\|^2_{A^\dagger A} -  \mathbb{E}[\|v_t - \theta^\star\|^2_{A^\dagger A}]$.
Note that there is no need to choose parameters $A_t$ and $B_t$ since only parameter $\alpha_t$ is required in this case.
\end{corollary}
In the convex case, we only have control over the objective function $F$ and not over the parameters. This in particular means that it is only possible to have guarantees on the dual objective in the case of non-smooth ADFS. 

\paragraph{Efficient iterations} Our extended APCG algorithm is also closely related with an arbitrary sampling version of APPROX~\cite{fercoq2015accelerated}. Similarly to Lee and Sidford~\cite{lee2013efficient}, APPROX also uses iterations that can be more efficient, especially in the linear case. These extensions can also be applied to APCG under the same assumptions, as shown in~\citet{lin2015accelerated}. We do not include the derivations in this paper since they are direct adaptations of the previously cited papers. Yet, the efficient formulations of the generalized APCG algorithm are presented in Appendix~\ref{app:efficient_apcg}.

\paragraph{Sampling with replacement}
The arbitrary sampling litterature for accelerated coordinate descent methods is vast~\citep{lee2013efficient,allen2016even,nesterov2017efficiency, hanzely2018accelerated}, and we present in this paper results for the general setting of block proximal coordinate gradient. Yet, standard mini-batch stochastic gradient descent algorithms use sampling \emph{with} replacement, whereas coordinate descent methods always use the notion of blocks, \emph{i.e.}, \emph{without} replacement. Algorithm~\ref{algo:generalized_apcg} does not extend to sampling with replacement, and this mainly comes from the fact that proximal updates do not mix well with sampling with replacement, and Lemma~\ref{lemma:lyapunov_psi} does not hold anymore in this case.

\section{Accelerated Decentralized Stochastic Algorithm}
\subsection{The dual problem}
\label{sec:model}
\begin{figure}[!htb]
    \centering
        \centering
        \vspace{-.4cm}
          \includegraphics[width=\linewidth]{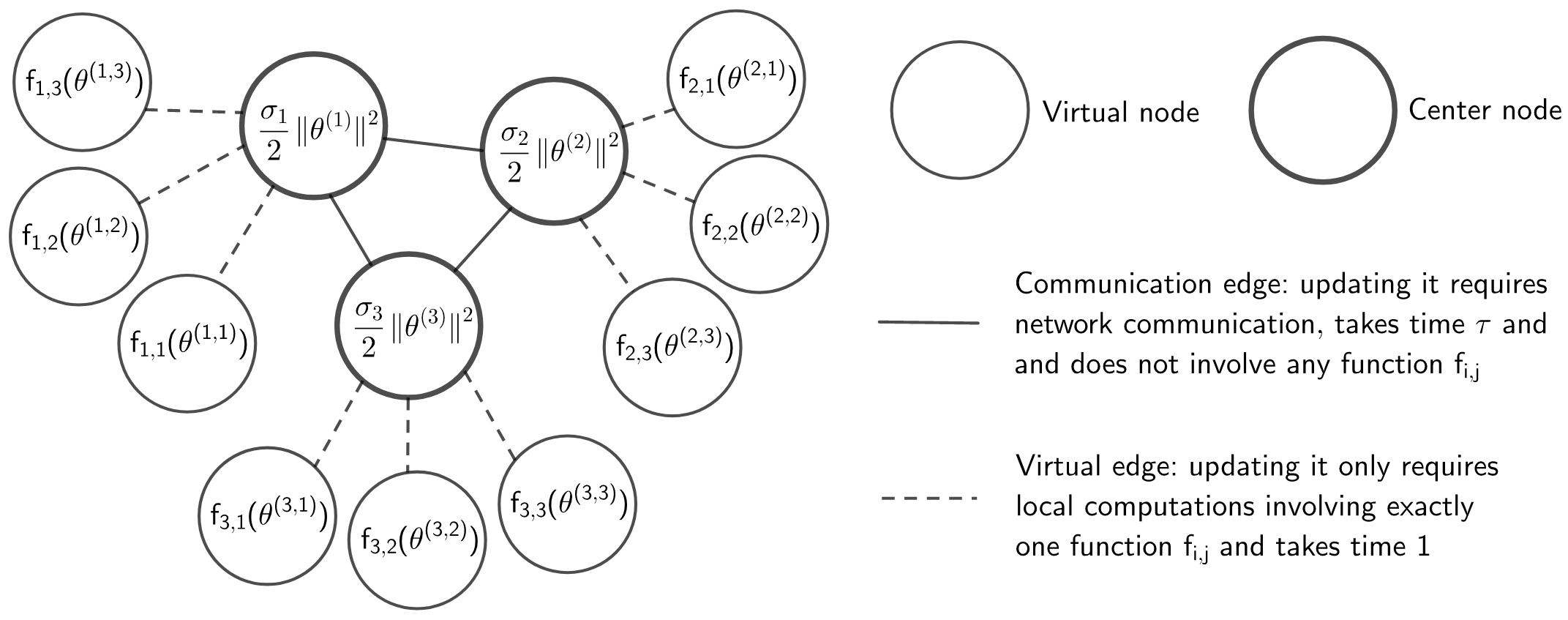}
  \vspace*{-.4cm}
\caption{Illustration of the augmented graph for $n=3$ and $m=3$.}
\label{fig:extended_graph}
\end{figure}
We now specify our approach to solve the problem of Equation~\eqref{eq:distributed_problem}. The first (classical) step consists in considering that all nodes have a local parameter, but that all local parameters should be equal because the goal is to have the global minimizer of the sum. Therefore, the problem writes: \vspace{-4pt}
\begin{equation}
    \label{eq:constr_distributed_prob}
    \min_{\theta \in \mathbb{R}^{n \times d}}\ \  \sum_{i=1}^n f_i(\theta^{(i)}) \ \  \text{ such that } \  \theta^{(i)} = \theta^{(j)} \text{ if } j \in \mathcal{N}(i),
\end{equation}
where $\mathcal{N}(i)$ represents the neighbors of node $i$ in the communication graph.
Then, \esdacd~and \msda~are obtained by applying accelerated (coordinate) gradient descent to an appropriate dual formulation of Problem~\eqref{eq:constr_distributed_prob}. In the dual formulation, constraints become variables and so updating a dual coordinate consists in performing an update along an edge of the network. In this work, we consider a new virtual graph in order to get a stochastic algorithm for finite sums. The transformation is sketched in Figure~\ref{fig:extended_graph}, and consists in replacing each node of the initial network by a star network. The centers of the stars are connected by the actual communication network, and the center of the star network replacing node~$i$ has the local function $f^{\rm comm}_i : x \mapsto \frac{\sigma_i}{2} \| x \|^2$. The center of node $i$ is then connected with $m$ nodes whose local functions are the functions $f_{i,j}$ for $j \in \{1, ..., m\}$. If we denote $E$ the number of edges of the initial graph, then the augmented graph has $n(1 + m)$ nodes and $E + nm$ edges. This augmented graph formulation was introduced in the conference version of this paper~\citep{hendrikx2019accelerated}.

Then, we consider one parameter vector $\theta^{(i,j)}$ for each function $f_{i,j}$ and one vector $\theta^{(i)}$ for each function~$f^{\rm comm}_i$. Therefore, there is one parameter vector for each node in the augmented graph. We impose the standard constraint that the parameter of each node must be equal to the parameters of its neighbors, but neighbors are now taken in the augmented graph. This yields the following minimization problem:\vspace{-5pt}
\begin{equation}
\label{eq:primal_constrained}
    \begin{split}
        \min_{\theta \in \mathbb{R}^{n(1 + m) d}} \ & \sum_{i=1}^n \bigg[ \ \ \sum_{j=1}^m f_{i,j}(\theta^{(i,j)}) + \frac{\sigma_i}{2} \| \theta^{(i)} \|^2 \bigg] \\
        \mbox{ such that } & {\theta^{(i)} = \theta^{(j)} \text{ if } j \in \mathcal{N}(i)}, 
       \text{ and } {\theta^{(i,j)} = \theta^{(i)} \ \ \forall j \in \{1, .., m\}}.
    \end{split}
\end{equation}

In the rest of the paper, we use letters $k,\ell$ to refer to any nodes in the augmented graph, and letters $i,j$ to specifically refer to a communication node and one of its virtual nodes. More precisely, we denote $(k, \ell)$ the edge between the nodes $k$ and $\ell$ in the augmented graph. Note that $k$ and $\ell$ can be virtual or communication nodes.
To clearly make the distinction between node variables and edge variables, for any vector on the set of nodes of the augmented graph $x\in \mathbb{R}^{n(1 + m)d}$ and for $k \in \{1, ..., n(1+m)\}$, we write $x^{(k)} \in \R^d$ (superscript notation) the subvector associated with node $k$. Similarly, for any vector on the set of edges of the augmented graph $\lambda \in \mathbb{R}^{(E + nm) d}$ and for any edge $(k, \ell)$ we write $\lambda_{k\ell} \in \R^d$ (subscript notation) the vector associated with edge $(k,\ell)$. For node variables, we use the subscript notation with a $t$ to denote time, for instance in Algorithm~\ref{algo:sc_adfs}. By a slight abuse of notations, we use indices $(i,j)$ instead of $(k,\ell)$ when specifically referring to virtual edges (or virtual nodes) and denote $\lambda_{ij}$ instead of $\lambda_{i,(i,j)}$ the virtual edge between node $i$ and node $(i,j)$ in the augmented graph. We note $e^{(k)} \in \R^n(1+ m)$ the unit vector associated with node $k$ and $e_{k\ell} \in \R^{E + nm}$ the unit vector associated with edge $k\ell$. We denote $M_1 \otimes M_2$ the Kronecker product between matrices $M_1$ and $M_2$.\vspace{3pt}

\paragraph{Constraints matrix} The constraints of Problem~\eqref{eq:primal_constrained} can be rewritten $A^\top \theta = 0$ in matrix form $A \in \mathbb{R}^{n(1 + m)d \times (nm + E)d}$ is such that for any $x\in \R^d$, 
$$A (e_{k\ell} \otimes x) = \mu_{k\ell} [(e^{(k)} - e^{(\ell)}) \otimes P_{k\ell} x],$$
for some $\mu_{k\ell} > 0$, and where $P_{k\ell}$ is a projector. For communication edges, we choose $P_{k\ell} = I_d$, and for communication edges, we choose $P_{ij}$ such that $f_{i,j}$ is $L_{ij}$-smooth with respect to $P_{ij}$, as defined in Equation~\eqref{eq:M_smoothness}. Note that this implies that $f_{ij}^*$ is $(1/L_{ij})$-strongly convex on $\Ker(P_{ij})^\perp$ and infinite elsewhere. The matrix $A$ is therefore completely defined by the $\mu_{k\ell}$ and the $M_{ij}$. Most results in the following sections heavily depend on the matrix $A$ and it is therefore very important to understand its structure. In particular, decentralized communications are defined by the matrix $A$. Indeed, $A$ can be understood as the canonical square root of the weighted Laplacian of the augmented graph. Similarly, if we note $A_{\comm} \in \R^{n \times E}$ the restriction of $A$ to non-virtual edges then $A_{\comm}$ is a square root of the weighted Laplacian of the communication graph. This is why a rescaled version of $A_{\comm}A_{\comm}^\top \in \R^{n \times n}$ is used as the gossip matrix in Algorithm~\ref{algo:sc_adfs}. To make things clearer, $A$ can be written as:
\begin{equation}
    A = \begin{pmatrix}
    A_\comm \otimes I_d & D_\mu \\
    0 & - D_\mu^{\rm diag}
    \end{pmatrix}, \hbox{ with }
\end{equation}

\begin{equation}
    D_\mu^{\rm diag} = \begin{pmatrix}
    \mu_{11} P_{11} & 0 & 0\\
    0 & \cdots & 0 \\
    0 & 0 & \mu_{nm} P_{nm}
    \end{pmatrix} \in \R^{nmd \times nmd}, \hbox{ and }
\end{equation}
\begin{equation}
    D_\mu = \begin{pmatrix}
    \mu_{11} P_{11} & \cdots & \mu_{1m} P_{1m} & 0 & 0 & 0 & 0\\
    0 & 0 & 0 & \cdots & 0 & 0 & 0 \\
    0 & 0 & 0 & 0 & \mu_{n1} P_{n1} & \cdots & \mu_{nm} P_{nm}
    \end{pmatrix} \in \R^{nd \times nmd}.
\end{equation}
All \emph{communication nodes} are linked by the true graph, whereas all \emph{virtual nodes} are linked to their corresponding communication node. Note that $A$ is defined differently in the conference paper~\citep{hendrikx2019accelerated}. Although the new definition of $A$ as an $n(m + 1)d \times (E + nm)d$ matrix is heavier in terms of notations, it allows to derive a much better communication complexity in some cases, for instance when $f_{ij}$ is a generalized linear model. Now that we have defined the matrix $A$ and emphasized its importance, we can write the dual formulation of the problem as:
\begin{equation}
    \max_{\lambda \in \mathbb{R}^{(nm + E) d}} - \sum_{i=1}^n \bigg[\sum_{j = 1}^m f_{i,j}^*\left((A\lambda)^{(i,j)}\right) + \frac{1}{2\sigma_i} \| (A\lambda)^{(i)} \|^2 \bigg],
\end{equation}
where the parameter $\lambda$ is the Lagrange multiplier associated with the constraints of Problem~\eqref{eq:primal_constrained}---more precisely, for an edge $(k,\ell)$, $\lambda_{k\ell} \in \mathbb{R}^d$ is the Lagrange multiplier associated with the constraint $\mu_{k\ell} P_{k\ell} (\theta^{(k)} - \theta^{(\ell)})=0$. This critically relies on the fact that $f_{i,j}^*(P_{ij}x) = f_{ij}^*(x)$ for all $x\in {\rm dom}(f_{ij}^*) = \Ker(P_{ij})^\perp$. At this point, the functions $f_{i,j}$ are only assumed to be convex (and not necessarily strongly convex) meaning that the functions $f_{i,j}^*$ are potentially non-smooth. This problem could be bypassed by transferring some of the quadratic penalty from the communication nodes to the virtual nodes before going to the dual formulation. Yet, this approach fails when $m$ is large because the smoothness parameter of $f_{i,j}^*$ would scale as $m / \sigma_i$ at best, whereas a smoothness of order $1 / \sigma_i$ is required to match optimal finite-sum methods. 
A better option is to consider the $f_{i,j}^*$ terms as non-smooth and perform proximal updates on them. The rate of proximal gradient methods such as APCG~\citep{lin2015accelerated} does not depend on the strong convexity parameter of the non-smooth functions $f_{i,j}^*$. Recall that each $f_{i,j}^*$ is $(1 / L_{i,j})$-strongly convex with respect to $P_{ij}$, 
so we can rewrite the previous equation in order to transfer all the strong convexity to the communication node. Noting that $(A\lambda)^{(i,j)} = - \mu_{ij} \lambda_{ij}$ when node $(i,j)$ is a virtual node associated with node $i$, we rewrite the dual problem as:
\begin{equation}
\label{eq:dual_problem}
    \min_{\lambda \in \mathbb{R}^{(E + nm) d}} q_A(\lambda) +  \sum_{i=1}^n \sum_{j=1}^m \psi_{ij}(\lambda_{ij}),
\end{equation}
with $\psi_{ij}: x \mapsto  \tilde{f^*_{ij}}(- \mu_{ij}x)$ and $ \tilde{f^*_{ij}}: x \mapsto f_{i,j}^*(x) - \frac{1}{2L_{i,j}}\|x\|^2_{P_{ij}}$ and $q_A:x \mapsto {\rm Trace}\big(\frac{1}{2}x^\top A^\top\Sigma^{\dagger} A x\big)$, where $\Sigma$ is the diagonal matrix such that the upper left (communication) block is equal to ${\rm diag}(\sigma_1, \cdots, \sigma_n) \otimes I_d$, and the rest of the diagonal is made of the blocks $L_{i,j} P_{ij}$ for the virtual node $(i,j)$. Since dual variables are associated with edges, using coordinate descent algorithms on dual formulations from a well-chosen augmented graph of constraints allows us to handle both computations and communications in the same framework. Indeed, choosing a variable corresponding to an actual edge of the network results in a communication along this edge, whereas choosing a virtual edge results in a local computation step. Then, we balance the ratio between communications and computations by simply adjusting the probability of picking a given kind of edges.

\subsection{The Algorithm: ADFS Iterations and Expected Error}
\label{sec:alg}
Recall that we would like to solve the problem of Equation~\eqref{eq:dual_problem}, which is to optimize the sum of a smooth and strongly convex term and of a non-smooth convex separable term. Proximal coordinate gradient algorithms are known to work well for these problems, which is why we would like to use APCG~\citep{lin2014accelerated}. Yet, the following points would lead to suboptimal rates if the standard APCG algorithm were used directly:
\begin{enumerate}
    \item The function $q_A$ is strongly-convex only on ${\rm Ker}(A)^\perp$.
    \item Picking blocks of coordinates is required to obtain a synchronous algorithm.  
    \item Choosing different probabilities for computation and communication coordinates is required to balance the ratio between communication and computation.
\end{enumerate}
These 3 points show the need for extending APCG and motivate our assumptions for Algorithm~\ref{algo:generalized_apcg}. Applying it to the problem of Equation~\eqref{eq:dual_problem} yields the general ADFS algorithm. We start by presenting the smooth version of ADFS in this section, and a non-smooth version is presented in Section~\ref{sec:non-smooth}. We denote $W_{k\ell} \in \mathbb{R}^{n(1 + m) \times n(1 + m)}$ the matrix such that $W_{k\ell} = (e^{(k)} - e^{(\ell)})(e^{(k)} - e^{(\ell)})^\top$ for any edge $(k, \ell)$. The previous section needed to consider the problem variables to be vectors in $\R^{n(1+m)d}$ in order to define the right matrix $A$. Yet, variables $x_t$, $y_t$ and $v_t$ from Algorithm~\ref{algo:sc_adfs} are variables associated with the nodes of the augmented graph and we will therefore consider them as matrices in $\mathbb{R}^{n(1 + m) \times d}$ (one row for each node) instead of vectors, which greatly simplifies notations. These variables are obtained by multiplying the dual variables of the proximal coordinate gradient algorithm applied to the dual problem of Equation~\eqref{eq:dual_problem} by $A$ on the left. We denote $\sigma_A = \lambda_{\min}^+(A^\top\Sigma^{\dagger}A)$ the smallest non-zero eigenvalue of the matrix $A^\top\Sigma^{\dagger}A$.

\begin{algorithm}
\caption{ADFS$\left(A, (\sigma_i), (L_{i,j}), (\mu_{k\ell}), (p_{k\ell}), \rho \right)$}
\label{algo:sc_adfs}
\begin{algorithmic}[1]
\STATE $\sigma_A = \lambda_{\min}^+(A^\top\Sigma^{\dagger}A)$, $\eta = \frac{\rho}{\sigma_A}$, $W_b = A_\comm P_b^\dagger A_\comm^\top$, $\tilde{W}_b = A_\comm P_b^\dagger A^\dagger_\comm$.
\STATE $x_0 = y_0 = v_0 = z_0 = 0^{(n + nm) \times d}$  \COMMENT{Initialization}
\FOR[Run for $K$ iterations]{$t=0$ to $K-1$}
\STATE $y_t = \frac{1}{1 + \rho}\left(x_t + \rho v_t\right)$
\STATE Sample block of edges $b$ \COMMENT{Edges sampled from the augmented graph}
\STATE $z_{t+1} = v_{t+1} = (1 - \rho) v_t + \rho y_t - \eta W_b\Sigma^{\dagger}y_t$ \COMMENT{Communication using $W_b$}
\IF{$b$ is a block of virtual edges}
\FOR{$i = 1$ to $n$}
\FOR{$j$ such that $(i,j) \in b$}
\STATE $v_{t+1}^{(i,j)} = {\rm prox}_{\eta \mu_{ij}^2 p_{ij}^{-1}\tilde{f}^*_{i,j}}\left(z_{t+1}^{(i,j)}\right)$  \COMMENT{Virtual node update using $f_{i,j}$}
\ENDFOR
\STATE $v_{t+1}^{(i)} = z_{t+1}^{(i)} + \sum_{j, (i,j) \in b} (z_{t+1}^{(i,j)} - v_{t+1}^{(i,j)})$ \COMMENT{Center node update}
\ENDFOR
\ENDIF
\STATE $x_{t+1} = y_t + \rho \tilde{W}_b(v_{t+1} - (1 - \rho) v_t - \rho y_t)$
\ENDFOR
\STATE \textbf{return} $\theta_K = \Sigma^{\dagger}v_K$ \COMMENT{Return primal parameter}
\end{algorithmic}
\end{algorithm}

\begin{theorem}
\label{thm:rate_adfs}
We denote $\theta^\star$ the minimizer of the primal function $F:x\mapsto \sum_{i=1}^n f_i(x)$ and $\theta^\star_A$ a minimizer of the dual function $F^*_A = q_A + \psi$. Then $\theta_t$ as output by Algorithm~\ref{algo:sc_adfs} verifies:
\begin{equation}
\label{eq:S}
    \esp{\|\theta_t - \theta^\star\|^2}
    \leq C_0 (1 - \rho)^t, \ \ \ \text{ if } \ \ \ \rho^2 \leq \min_{b} \frac{\lambda_{\min}^+ (A^\top \Sigma^{\dagger} A )}{\lambda_{\max}(A^\dagger A P_b^\dagger A^\top \Sigma^{\dagger} A P_b^\dagger A^\dagger A)},
\end{equation}
with $C_0 = \lambda_{\max}(A^\top\Sigma^{-2}A) \left[\|A^\dagger A\theta^\star_A\|^2 + 2 \sigma_A^{-1} \left(F^*_A(0) - F^*_A(\theta^\star_A)\right)\right]  $.
\end{theorem}

We now quickly discuss the convergence rate of ADFS, and present the basic derivations required to obtain Algorithm~\ref{algo:sc_adfs}, as well as the proof Theorem~\ref{thm:rate_adfs}. 

\paragraph{Convergence rate} The parameter $\rho$ controls the convergence rate of ADFS. It is defined by the minimum of the individual rates for each block, which involves the spectrum of a product of matrices related to the regularity of the local functions ($\Sigma^{\dagger}$), to the graph ($A$) and to the sampling scheme ($P_b^\dagger$). Note that Theorem~\ref{thm:rate_adfs} recovers the asynchronous version of ADFS~\cite{hendrikx2019accelerated} if only one coordinate is sampled at each step. Relations are more complex in the general case, which is why simple scalar expressions are replaced by the spectrum of products of matrices in this paper. In Section~\ref{sec:perfs}, we carefully choose the free parameters $\mu_{k\ell}$ and $p_{k\ell}$ to get the best convergence speed.

\paragraph{Projection of virtual edges} We need to verify that Assumption~\ref{assumption:gen_apcg} is respected in order to be able to use Theorem~\ref{thm:gen_apcg} to derive Theorem~\ref{thm:rate_adfs}. In particular, for any edge $(k, \ell)$, either the proximal part $\psi_{k\ell} = 0$ or the dual coordinate is such that for all $\theta \in \R^d$, $(e_{k\ell}^\top \otimes \theta) A^\dagger A (e_{k\ell} \otimes \theta) = 1$, which is equivalent to having $A^\dagger A (e_{k\ell} \otimes \theta) = (e_{k\ell} \otimes \theta)$. In our case, $\psi_{k\ell} = 0$ when $(k,\ell)$ is a communication edge. The condition is actually not verified for virtual edges in our formulation since we introduce the projectors $P_{ij}$. Yet, we do not need this to hold for any $\theta \in \R^d$. Indeed, the updates of Algorithm~\ref{algo:sc_adfs} are such that $v_t^{(ij)} \in \Ker(P_{ij})^\perp$ for all $t$ and $(i,j)$, so we only need $A^\dagger A (e_{ij} \otimes \theta) = (e_{ij} \otimes \theta)$ to hold for $\theta \in \Ker(P_{ij})^\perp$. Lemma~\ref{lemma:acrossa} shows that the projection condition is satisfied by virtual edges.
\begin{lemma}
\label{lemma:acrossa}
$A^\dagger A (e_{ij} \otimes \theta) = e_{ij} \otimes \theta$ for all virtual edges $(i,j)$ and $\theta \in \Ker(P_{ij})^\perp$.
\end{lemma}
\begin{proof}
Let $\theta \in \Ker(P_{ij})^\perp$, and $x \in \mathbb{R}^{E + nm}$ such that $A(x \otimes \theta) = 0$. From the definition of $A$, either $x = 0$ or the support of $x$ is a cycle of the graph. Indeed, for any edge $(k,\ell)$, $A (e_{k\ell} \otimes \theta)$ has non-zero weights only on nodes $k$ and $\ell$. Virtual nodes have degree one, so virtual edges are part of no cycles and therefore $x^\top e_{k,\ell} = 0$ for all virtual edges $(k,\ell)$. Operator $A^\dagger A$ is the projection operator on the orthogonal the kernel of $A$, so it is equal to $P_{ij}$ on virtual edges, and $P_{ij} \theta = \theta$.
\end{proof}

\paragraph{Obtaining Line 6} The form of the communication update (virtual or not) of line 6 in Algorithm~\ref{algo:sc_adfs} comes from the fact that the update of block $b$ writes $A P_b^\dagger \nabla q_A(y_t) = A P_b^\dagger A^\top \Sigma^{\dagger} y_t = W_b \Sigma^{\dagger} y_t$.

\paragraph{Obtaining the proximal formulation of Lines 10 and 11} Algorithm~\ref{algo:sc_adfs} is obtained by directly applying Algorithm~\ref{algo:generalized_apcg} on the dual problem of Equation~\eqref{eq:dual_problem}. Then, all lines are multiplied by $A$ on the left in order to switch from dual variables in $\R^{E + nm}$ associated with edges to primal variables in $\R^{n + nm}$ associated with nodes, which is a standard transformation~\citep{scaman2017optimal,hendrikx2018accelerated}. Yet, APCG uses a proximal step, which is a non-linear operation, and the transformation is not straightforward in this case. We now present the derivations leading to Algorithm~\ref{algo:sc_adfs}, which are the same as in the conference version~\citep{hendrikx2019accelerated}. More specifically, we note $\tilde{v}_t \in \R^{E + nm}$ the dual variable and $v_t = A\tilde{v}_t$ the primal variable of Algorithm~\ref{algo:sc_adfs}. We use the same notations for the other variables. We know from applying APCG that:
\begin{equation}
\label{eq:dual_prox}
    \tilde{v}_{t+1}^{(i,j)} = {\rm prox}_{\eta p_{ij}^{-1} \psi_{ij}}\left(\tilde{z}_{t+1}^{(i,j)}\right) .
\end{equation}
Since $(i,j)$ are coordinates associated with virtual edges, we also know that $(A\tilde{w}_t)^{(i,j)} = - \mu_{ij}\tilde{w}_t^{(i,j)}$. Therefore, Equation~\eqref{eq:dual_prox} can be rewritten as: 
\begin{equation}
    v_{t+1}^{(i,j)}  = - \mu_{ij} {\rm prox}_{\eta p_{ij}^{-1} \psi_{ij}}\left(-\frac{1}{\mu_{ij}}z_{t+1}^{(i,j)} \right), 
\end{equation}
which only involves primal variables. We can further rewrite this equation in the simpler form of Line 10 by using a change of variables to write that:
\begin{align*}
    - \mu_{ij} \arg\min_v &\frac{1}{2\eta p_{ij}^{-1}}\|v - \left(- \frac{1}{\mu_{ij}} z_{t+1}^{(i,j)} \right)\|^2 +  \tilde{f^*_{ij}}(-\mu_{ij}v) \\ &=\arg\min_{\tilde{v}} \frac{1}{2\eta p_{ij}^{-1}\mu_{ij}^2}\|\tilde{v} -  z_{t+1}^{(i,j)} \|^2 + \tilde{f^*_{ij}}(\tilde{v}).
\end{align*}

Finally, Line 11 is obtained by remarking that the proximal step is only performed for coordinates associated with virtual edges and therefore that since $v_{t+1}^{(i,j)} = z_{t+1}^{(i,j)}$ if $i \notin b_t$,
$$v_{t+1}^{(i)} - z_{t+1}^{(i)} = \sum_{j \in \mathcal{N}(i)} \mu_{i,j} (\tilde{v}_{t+1}^{(i,j)} - z_{t+1}^{(i,j)}) = - \sum_{j, \ (i,j) \in b_t} (v_{t+1}^{(i,j)} - z_{t+1}^{(i,j)}).$$

We have justified in the remarks above that Algorithm~\ref{algo:sc_adfs} is indeed the direct application of the efficient implementation of Algorithm~\ref{algo:generalized_apcg} to the Problem of Equation~\eqref{eq:dual_problem}, which verifies Assumption~\ref{assumption:gen_apcg}. Then, Theorem~\ref{thm:rate_adfs} is a corollary of Corollary~\ref{corr:sc_apcg}, as shown below.

\begin{proof}[Proof of Theorem~\ref{thm:rate_adfs}]
Following \citet{lin2015accelerated}, and noting $q: x \mapsto \frac{1}{2} x^\top \Sigma^{\dagger} x$, the primal optimal point $\theta^\star$ can be retrieved as $\theta^\star = \nabla q (A\theta^\star_A) = \Sigma^{\dagger}A\theta^\star_A$, where $\theta^\star_A$ is the optimal dual parameter. Finally, $$\lambda_{\max}(A^\top\Sigma^{-2}A)^{-1} \|\theta_t - \theta^\star\|^2 \leq \lambda_{\max}(A^\top\Sigma^{-2}A)^{-1} \|\Sigma^{\dagger}A(\tilde{\theta}_t - \theta^\star_A)\|^2 \leq \|\tilde{\theta}_t - \theta^\star_A\|^2_{A^\dagger A},$$ where $\tilde{\theta}_t = \phi^{K+1} \tilde{u}_t + \tilde{z}_t$. Finally, the control on $\|\tilde{\theta}_t - \theta^\star_A\|^2_{A^\dagger A}$ is given by Corollary~\ref{corr:sc_apcg}. Note that APCG also gives a guarantee in terms of dual function values but we drop it in order to have a simpler statement. 
\end{proof}

\subsection{Implementation details}

We discuss several aspects related to the implementation of Algorithm~\ref{algo:sc_adfs} below, and provide its Python implementation in supplementary material.

\paragraph{Primal proximal updates}
The proximal step of Line 10 is performed with the function $\tilde{f^*}_{i,j}: x \rightarrow f_{i,j}^*(x) - \frac{1}{2L_{i,j}}\|x\|^2$ instead of $f_{i,j}$. Yet, Moreau identity~\citep{parikh2014proximal} provides a way to retrieve the proximal operator of $f^*$ using the proximal operator of $f$, but this does not directly apply to $\tilde{f}^*_{i,j}$, making its proximal update hard to compute when no analytical formula is available to compute $\tilde{f}^*_{i,j}$. Fortunately, the proximal operator of $\tilde{f}^*_{i,j}$ can be retrieved from the proximal operator of $f^*_{i,j}$. Following the derivations from the conference paper~\citep{hendrikx2019accelerated}, we now show how to implement Algorithm~\ref{algo:sc_adfs} in a primal-only way. More specifically, if we denote $\tilde \eta_{ij} = \eta \mu_{ij}^2 p_{ij}^{-1}$ then for any $x \in \R^{n + nm}$, we can also express the update only in terms of $f_{i,j}^*$:
\begin{align*}
    {\rm prox}_{\tilde{\eta}_{ij} \tilde{f}^*_{i,j}}\left(x\right) &= \arg\min_v \frac{1}{2\tilde{\eta}_{ij}} \|v - x\|^2 + f^*_{i,j}(v) - \frac{1}{2L_{i,j}}\|v\|^2\\
    &= \arg\min_v \frac{1}{2}\left(\tilde{\eta}_{ij}^{-1} - L_{i,j}^{-1}\right)\|v\|^2 - \tilde{\eta}_{ij}^{-1} v^\top x + f^*_{i,j}(v) \\
    &= \arg\min_v \frac{1}{2\left(\tilde{\eta}_{ij}^{-1} - L_{i,j}^{-1}\right)^{-1}}\|v - \left(1 - \tilde{\eta}_{ij} L_{i,j}^{-1}\right)^{-1} x\|^2 + f^*_{i,j}(v) \\
    &= {\rm prox}_{\left(\tilde{\eta}_{ij}^{-1} - L_{i,j}^{-1}\right)^{-1} f^*_{i,j}}\left(\left(1 - \tilde{\eta}_{ij} L_{i,j}^{-1}\right)^{-1} x\right).
\end{align*}
Then, we use the identity:
\begin{equation}
    {\rm prox}_{\left(\eta f\right)^*}(x) = \eta {\rm prox}_{\eta^{-1} f^*}\left(\eta^{-1} x\right),
\end{equation}
and the Moreau identity leads to:
\begin{equation}
    {\rm prox}_{\eta f^*}(x) = x -  \eta {\rm prox}_{\eta^{-1} f}\left(\eta^{-1} x\right).
\end{equation}

This allows us to retrieve the proximal operator on $\tilde{f}^*_{i,j}$ using only the proximal operator on $f_{i,j}$:
\begin{equation}
    \left(1 - \tilde{\eta}_{ij} L_{i,j}^{-1}\right) {\rm prox}_{\tilde{\eta}_{ij} \tilde{f}^*_{i,j}}\left(x\right) = x - \tilde{\eta}_{ij} {\rm prox}_{\left(\tilde{\eta}_{ij}^{-1} - L_{i,j}^{-1}\right) f_{i,j}} \left(\tilde{\eta}_{ij}^{-1} x\right).
\end{equation}

Note that the previous calculations are valid as long as  $\tilde{\eta}_{ij} L_{i,j}^{-1} \leq 1$ for all virtual edges. Using the same values for $\mu_{ij}^2$ as in Assumption~\ref{assumption:parameters_choice}, and using the fact that $\eta = \rho / \sigma_A = 2\rho / \alpha$, this condition writes $ 2\rho \leq p_{ij}$ for all virtual edges $(i,j)$. By definition of $\rho$ we have $\rho \leq p_{ij} / \sqrt{2(1 + L_{ij}/\sigma_i)}$, so this constraint simply makes $\rho$ smaller by a $\sqrt{2}$ factor in the worst case (and does not change anything as long as $L_{ij} > \sigma_i$ for all $j$). In the case of Algorithm~\ref{algo:sc_adfs}, the update of Line 10 can be rewritten:
$$v_{t+1}^{(i,j)} = \left(\tilde{\eta}_{ij}^{-1} - L_{ij}^{-1}\right)^{-1}\left[\tilde{\eta}_{ij}^{-1}z_{t+1}^{(i,j)} - \prox_{(\tilde{\eta}_{ij}^{-1} - L_{ij}^{-1})f_{i,j}}\left( \tilde{\eta}_{ij}^{-1} z_{t+1}^{(i,j)}\right) \right].$$

\paragraph{Communications}
Communications in Algorithm~\ref{algo:sc_adfs} are abstracted by multiplication by the Matrix $W_b$ for a given batch of coordinates $b$. Note that if $b$ is a set of virtual edges then no communications in the network are required since $W_b$ only requires information exchange between central nodes and their virtual nodes. The formulation of ADFS suggests that another communication round using the matrix $\tilde{W_b}$ is required for the actual update. Yet, if $b_t$ is such that $P_b^\dagger A^\dagger A P_b^\dagger$ is a diagonal matrix, then $\tilde{W}_b W_b$ can be performed with only one round of communications. This is the case for example if $b_t$ is a set of virtual edges (then no communications are actually required). If $b_t$ is the set of all communication edges, then $\tilde{W}_{b_t} = P_{b_t}^\dagger$ since $\pcomm P_{b_t}^\dagger$ is the identity on ${\rm Ker}(A)^\perp$ so no extra communication is required in this case either.\vspace{3pt}

\paragraph{Sparse updates} The way sequences $u_{t+1}$ and $z_{t+1}$ are updated means that the only nodes that are updated at time $t$ are those for which $(I - \rho \tilde{W}_{b_t})h_t$ is non-zero. Since $h_t$ is very sparse, it in particular means that the parameters of virtual nodes only need to be updated when their function is needed for an update. This would not be the case if we had used the formulation of Algorithm~\ref{algo:generalized_apcg} directly.\vspace{3pt}

\paragraph{Linear case}
For many standard machine learning problems, $f_{i,j}(\theta) = \ell(X_{i,j}^\top\theta)$ with $X_{i,j} \in \mathbb{R}^d$. This implies that $f_{i,j}^*(\theta) = + \infty$ whenever $\theta \notin {\rm Span}\left(X_{i,j}\right)$. Therefore, the proximal steps on the Fenchel conjugate only have support on $X_{i,j}$, meaning that they are one-dimensional problems that can be solved in constant time using for example the Newton method when no analytical solution is available. Warm starts (initializing on the previous solution) can also be used for solving the local problems even faster so that in the end, a one-dimensional proximal update is only a constant time slower than a gradient update. Note that this also allows to store parameters $v_t$ and $y_t$ as  scalar coefficients for virtual nodes, thus greatly reducing the memory footprint of ADFS. Finally, the projectors are equal to $P_{ij} = X_{ij} X_{ij}^\top / \|X_{ij}\|^2$ in this case which, as we will see, implies that $\kappa_\comm = \kappa_b$ when $\sigma_i = \sigma$ for all $i$.\vspace{3pt}

\paragraph{Sparse Communications} The communications in Algorithm~\ref{algo:sc_adfs} require sending the full local vector of each node, which can be very expensive when the dimension of the dataset is very high. This limitation is quite hard to bypass since ADFS relies on model averaging. One trick that can be used was introduced along with the DSBA algorithm~\citet{shen2018towards}, and consists in transmitting the updates $h_t$ instead of the local parameter. The other nodes can then emulate the updates as long as they know the gossip matrix. Yet, this approach increases local computations and storage by a significant margin (all nodes need to keep track of and compute the parameters of the other nodes in the network). Besides, it requires extremely sparse datasets to yield significant gains since all nodes eventually need to know the full increment $h_t$. Yet, this can be useful when communications are more frequent than computations or when the network is small. We do not elaborate more on this trick since it can be directly adapted from the original DSBA paper~\citet{shen2018towards}. \vspace{3pt}

\paragraph{Unbalanced local datasets} We assume that all local datasets are of fixed size $m$ in order to ease reading. Yet, the impact of the value of $m$ on Algorithm~\ref{algo:sc_adfs} is indirect, and unbalanced datasets can be handled without any change. \vspace{3pt}

\paragraph{Natural Strong Convexity} ADFS is derived when strong convexity is obtained through L2 regularization. It is possible to generalize this to arbitrary strongly convex functions $\omega_i$ by simply replacing $\frac{1}{2\sigma_i}\| \cdot \|^2$ by $\omega_i^*$, and performing the same derivations. Yet, we chose to focus on the L2 regularization case to ease reading of the paper. \vspace{3pt}

\subsection{Non-smooth setting}
The accelerated proximal coordinate gradient algorithm can be applied to the problem of Equation~\eqref{eq:dual_problem} even if the function $q_A$ is not strongly convex on ${\rm Ker}(A)^\perp$. This is for example the case when the functions $f_{i,j}$ are not smooth so that $\Sigma^{\dagger}$ has diagonal blocks equal to 0 and therefore ${\rm Ker}(A^\top\Sigma^{\dagger}A) \not\subset {\rm Ker}(A)$ so $\sigma_A = 0$. In this case, the choice of coefficients from Corollary~\ref{corr:cvx_apcg} leads to Algorithm~\ref{algo:ns_adfs}, a formulation of ADFS that provides error guarantees when primal functions $f_{i,j}$ are not smooth. More formally, if we define $F^*: x \rightarrow \sum_{i=1}^n \bigg[\sum_{j = 1}^m f_{i,j}^*\left(x^{(i,j)}\right) + \frac{1}{2\sigma_i} \| x^{(i)} \|^2 \bigg]$, then, we have:
\begin{theorem}
\label{thm:adfs_non_smooth}
If the functions $f_{i,j}$ are non-smooth then NS-ADFS guarantees:
$$ \esp{F^*(x_t)} - F^*(\theta^\star) \leq \frac{2}{t^2} \left[\frac{S^2}{\lambda_{\min}^+(A^\top A)} r_t^2 + \frac{6}{p_{\min}^2}\left[F^*(x_0) - F^*(\theta^\star)\right]\right],$$
with $S^2 = \max_b \lambda_{\max}(A^\dagger A P_b^\dagger A^\top \Sigma^{\dagger} A P_b^\dagger A^\dagger A)$, $r_t^2 = \|v_0 - \theta^\star\|^2 - \|v_t - \theta^\star\|^2$ and $p_{\min}$ is taken over virtual edges.
\end{theorem}
The guarantees provided by Theorem~\ref{thm:adfs_non_smooth} are weaker than in the smooth setting. In particular, we lose linear convergence and get the classical  accelerated sublinear $O(1/t^2)$ rate. We also lose the bound on the primal parameters--- recovering primal guarantees is beyond the scope of this work. Note that the extra $\lambda_{\min}^+(A^\top A)$ term comes from the fact that Theorem~\ref{thm:adfs_non_smooth} is formulated with primal parameter sequences $x_t = A \tilde x_t$. Also note that $\alpha_t = \grando{t^{-1}}$.

\begin{algorithm}
\caption{NS-ADFS}
\label{algo:ns_adfs}
\begin{algorithmic}[1]
\STATE $\alpha_0 = \min_{\text{virtual edges } (i,j)} p_{ij}$, $\eta_{t} = \frac{1}{\alpha_t S^2}$, $W_b = A_\comm P_b^\dagger A_\comm^\top$, $\tilde{W}_b = A_\comm P_b^\dagger A^\dagger_\comm$, $x_0 = 0$, $v_0 = 0$, $t = 0$ \COMMENT{Initialization}
\WHILE{$t < T$}
\STATE $y_t = (1 - \alpha_t)x_t + \alpha_t v_t$
\STATE Sample block of edges $b$ \COMMENT{Edges sampled from the augmented graph}
\STATE $v_{t+1} = z_{t+1} = v_t - \eta_t W_b \Sigma^{\dagger}y_t$
\COMMENT{Communication abstracted by the matrix $W_b$}
\IF{$b$ is a block of virtual edges}
\FOR{$i = 1$ to $n$}
\FOR{$j$ such that $(i,j) \in b$}
\STATE $v_{t+1}^{(i,j)} = {\rm prox}_{\eta_t \mu_{ij}^2 p_{ij}^{-1} f^*_{i,j}}\left(z_{t+1}^{(i,j)}\right)$
\STATE $v_{t+1}^{(i)} = z_{t+1}^{(i)} + z_{t+1}^{(i,j)} - v_{t+1}^{(i,j)}$
\ENDFOR
\ENDFOR
\ENDIF
\STATE $x_{t+1} = y_t + \alpha_t \tilde{W}_b(v_{t+1} - v_t)$
\STATE $\alpha_{t+1} = \frac{\sqrt{\alpha_t^4 + 4\alpha_t^2} - \alpha_t^2}{2}$
\ENDWHILE
\STATE \textbf{return} $\theta_t = \Sigma^{\dagger}v_t$
\end{algorithmic}
\end{algorithm}

\subsection{Performances and Parameters Choice in the Homogeneous Setting}
\label{sec:perfs}
We now prove the time to convergence of \adfs~presented in Table~\ref{fig:table_speeds}, and detail the conditions under which it holds. Indeed, Section~\ref{sec:alg} presents \adfs~in full generality but the different parameters have to be chosen carefully to reach optimal speed. In particular, we have to choose the coefficients $\mu$ to make sure that the graph augmentation trick does not cause the smallest positive eigenvalue of $A^\top \Sigma^{\dagger} A$ to shrink too much, which is done by Lemma~\ref{lemma:main_sigma_A_lb}. 

\begin{assumption}[Parameters choice]
\label{assumption:parameters_choice}
For arbitrary $\mu_{k\ell}$ and for all communication edges, we denote $L = A_{\rm comm} A_{\rm comm}^\top \in \mathbb{R}^{n \times n}$ the Laplacian of the communication graph. Let $D_M$ and $\tilde{D}_M$ be the diagonal matrices such that $(D_M)_{ii} = \sigma_i + \lambda_{\max}\left(\sum_{j=1}^m L_{i,j} P_{ij}\right)$ and $(\tilde{D}_M)_{ii} = \sigma_i + 2\lambda_{\max}\left(\sum_{j=1}^m L_{i,j} P_{ij}\right)$. The local condition number of node $i$ is $\kappa_i = (D_M)_{ii} / \sigma_i$, and we choose the weights of virtual edges as $\mu_{ij}^2 = \alpha L_{i,j}$, with $\alpha = 2 \lambda_{\min}^+(A_\comm^\top \tilde{D}_M^{-1} A_\comm)$, and their probabilities as $p_{ij} = p_{\rm comp} (1 + L_{i,j} \sigma_i^{-1})^{\frac{1}{2}} / S_i$ with $S_i = \sum_{j=1}^m (1 + L_{i,j} \sigma_i^{-1})^{\frac{1}{2}}$ the normalizing constant for node $i$.
\end{assumption}

This choice of parameters allows to tightly bound $\lambda_{\min}^+(A^\top\Sigma^{\dagger}A)$, which defines the rate of convergence of ADFS.

\begin{lemma}
\label{lemma:main_sigma_A_lb} 
If Assumption~\ref{assumption:parameters_choice} holds, then for any $x \in \R^{E + nm}$ we have
$$\|x\|^2_{A^\top \Sigma^{\dagger} A} \geq \lambda_{\min}^+(A_\comm^\top \tilde{D}_M^{-1} A_\comm) \|x\|^2_{A^\dagger A}.$$
In particular, $\sigma_A \geq \alpha/2$.
\end{lemma}

\begin{proof}[Proof sketch]
The proof studies the Schur complement of $\Sigma^{-\frac{1}{2}} AA^\top \Sigma^{-\frac{1}{2}}$. This yields a characterization of the eigenvalues of $A^\top \Sigma^{\dagger}A$ in terms of a determinant equation of the form $\det(L_\comm - \Delta_\lambda) = 0$, with $\Delta_\lambda$ a block-diagonal matrix that depends on $\lambda$ and $L_\comm = A_\comm A_\comm^\top$, where $A_\comm \in \R^{nd \times Ed}$ is the restriction of $A$ to communication nodes and edges. Then, Lemma~\ref{lemma:ker_u_v} gives necessary conditions for an $x$ to be in ${\rm Ker}(L_\comm - \Delta_\lambda)$, and we thus deduce bounds on the smallest eigenvalue of $A^\top \Sigma^{\dagger}A$ from upper bounds on $\Delta_\lambda$. Note that the proof is simpler than in the conference paper~\citep{hendrikx2019accelerated}, and the different choices in Assumption~\ref{assumption:parameters_choice} allow for a tighter bound.
\end{proof}

We now study parameter $\rho$ more in details, which is defined in Equation~\eqref{eq:S} by bounding the spectrum of a matrix that depends on the block of coordinates chosen. The spectral properties of this matrix heavily depend on whether the block contains actual communication edges or virtual edges. One can trade $p_{\rm comp}$ for $p_{\rm comm}$ so that the bound is the same for both kind of edges. This amounts to tuning the ratio between communications and computations. We first make some assumptions on the sampling performed, and then detail the communication and computation rate under this sampling. 

\begin{assumption}[Synchronous sampling]
\label{assumption:synchronous_sampling}
The sampling of edges is such that:
\begin{itemize}
    \item With probability $\pcomm$, $b_t = b_{\rm comm}$, the set of all communication edges. This corresponds to communicating over all edges of the network, which comes down to a multiplication by the gossip matrix $L$.
    \item With probability $\pcomp = 1 - \pcomm$, a computation step is performed. In this case, $b_t = \{ (i, j_t(i)), \ i \in \{1, ..., n\}\}$, where $j_t(i) = j$ with probability $p_{i,j}$. This corresponds to each node sampling exactly one virtual edge. 
\end{itemize}
\end{assumption}

This synchronous sampling defines the blocks of coordinates $b_t$ that are picked by Algorithm~\ref{algo:sc_adfs}. It is then possible to compute $\rho$, the rate of convergence of ADFS, depending on the frequency of communication $\pcomm$.

\begin{lemma}
\label{lemma:rate_adfs}
We denote $\kappa_s = \max_i \kappa_i$ and $\gamma$ the spectral gap of the Laplacian of the communication graph $L_\comm = A_\comm A_\comm^\top$. Under the synchronous sampling of Assumption~\ref{assumption:synchronous_sampling}, the convergence rate of ADFS is such that
\begin{align*}
    &\rho^2 = \min\left( \frac{\gamma}{\kappa_\comm}\pcomm^2 , \frac{\pcomp^2}{2(m + \sqrt{m\kappa_s})^2}\right), \hbox{ with }\\
    &\kappa_\comm = \frac{\lambda_{\max}(A_\comm^\top \Sigma_\comm^{-1} A_\comm) \ / \ \lambda_{\max}(A_\comm^\top A_\comm)}{\lambda_{\min}^+(A_\comm^\top \tilde{D}_M^{-1} A_\comm) \ / \ \lambda_{\min}^+(A_\comm^\top A_\comm) }.
\end{align*}
If $\sigma_k = \sigma$ for all $k$ (homogeneous case) then $\kappa_\comm = \max_i (\tilde{D}_M)_{ii} / \sigma$. If $f_{ij}(\theta) = g(X_{ij}^\top \theta)$ and $g$ is $L_g$-smooth then $(\tilde{D}_M)_{ii} = \sigma_i + 2 L_g \lambda_{\max}\left(\sum_{j=1}^m X_{ij} X_{ij}^\top\right)$. Therefore, $\kappa_\comm$ is of order $\kappa_b$ rather than $\kappa_s$, and we recover the expected communication complexity for decentralized algorithms.
In heterogeneous cases, $\kappa_\comm$ better captures the relations between the regularity of the local functions and the topology of the communication graph. 
\end{lemma}

Now that we have specified the rate of ADFS (improvement per iteration), the only step left is to tune $\pcomm$ to minimize time needed to reach a given precision $\varepsilon$. Theorem~\ref{thm:adfs_speed_precise} gives a choice of $\pcomm$ that achieves optimal rates. 

\begin{theorem}
\label{thm:adfs_speed_precise}
If $p_{\rm comm} = \left(1 + \sqrt{\frac{2\gamma}{\kappa_\comm}}(m + \sqrt{m\kappa_s})\right)^{-1}$, then running Algorithm~\ref{algo:sc_adfs} for $K = \rho^{-1}\log\left(\varepsilon^{-1}\right)$ iterations guarantees $\mathbb{E}\left[\| \theta_K - \theta^\star\|^2 \right] \leq C_0 \varepsilon$, and takes time $T(K)$, with $T(K)$ such that:
\begin{equation*}
    \esp{T(K)} \leq \left(\sqrt{2}(m + \sqrt{m\kappa_s}) + \tau \sqrt{\frac{\kappa_\comm}{\gamma}} \right) \log\left(\frac{1}{\varepsilon}\right).
\end{equation*}
\end{theorem}

\begin{remark}[Tightness of the bound]
For generalized linear models with homogeneous regulatization, we already saw that $\kappa_\comm = O(\kappa_b)$, and thus ADFS is optimal. Yet, the function used to derive the lower bound does not have $\sigma_k = \sigma_\ell$ for all $k, \ell$, so we cannot directly say that $\kappa_\comm = \kappa_b$ in this case. Fortunately, it is possible to exploit the structure of the graph and of $\tilde{D}_M$ and $\Sigma^{\dagger}_\comm$ to derive that $\kappa_\comm = O(\kappa_b)$ anyway. Detailed derivations are presented in Appendix~\ref{app:adfs_line_graph}.
\end{remark}

\subsection{Non-smooth setting}
\label{sec:non-smooth}
The leading constant governing the convergence rate of ADFS in the non-smooth case is $\lambda_{\min}^+\left(A^\top A\right) / S^2$, which is very related to the constant for the smooth case. Indeed, $\lambda_{\min}^+\left(A^\top\Sigma^{\dagger}A\right)$ is simply replaced by $\lambda_{\min}^+\left(A^\top A\right)$. In particular, we can use the results of Lemma~\ref{lemma:main_sigma_A_lb} and simply replace $\Sigma^{\dagger}$ by the identity matrix. In this case, we get $\mu_{ij}^2 = \frac{\lambda_{\min}^+(L)}{1 + m}$ when $(i,j)$ is a computation edge, which yields $$\lambda_{\min}^+(A^\top A) \geq \frac{\lambda_{\min}^+(L)}{2(m + 1)}.$$
Similarly, it is possible to set $p_{i,j} = \pcomp / m$ and get a non-smooth equivalent of Lemma~\ref{lemma:rate_adfs} by writing:
$$S^2 \leq \frac{1}{\sigma_{\min}}\max\left(\frac{\lambda_{\max}(L)}{\pcomm^2}, \frac{\lambda_{\min}^+(L) m^2}{(m + 1)\pcomp^2}\right),$$
and so 
$$ \frac{S^2}{\lambda_{\min}^+(A^\top A)} \leq \frac{2(m + 1)}{\sigma_{\min}}\max\left(\frac{1}{ \gamma \pcomm^2}, \frac{m}{\pcomp^2}\right). $$
Yet, there is no linear convergence and the precise optimization of $\pcomm$ depends on the leading term from the bound of Theorem~\ref{thm:adfs_non_smooth}. If the term proportional to $r_t^2$ dominates then the same arguments as those of Theorem~\ref{thm:adfs_speed_precise} can be applied and the optimal choice is $p_{\rm comm} = \left(1 + \sqrt{\gamma m}\right)^{-1}$.

\section{Experiments}
\label{sec:experiments}
\begin{figure*}
\centering
\begin{subfigure}{.32\textwidth}
  \centering
  \includegraphics[width=1.0\linewidth]{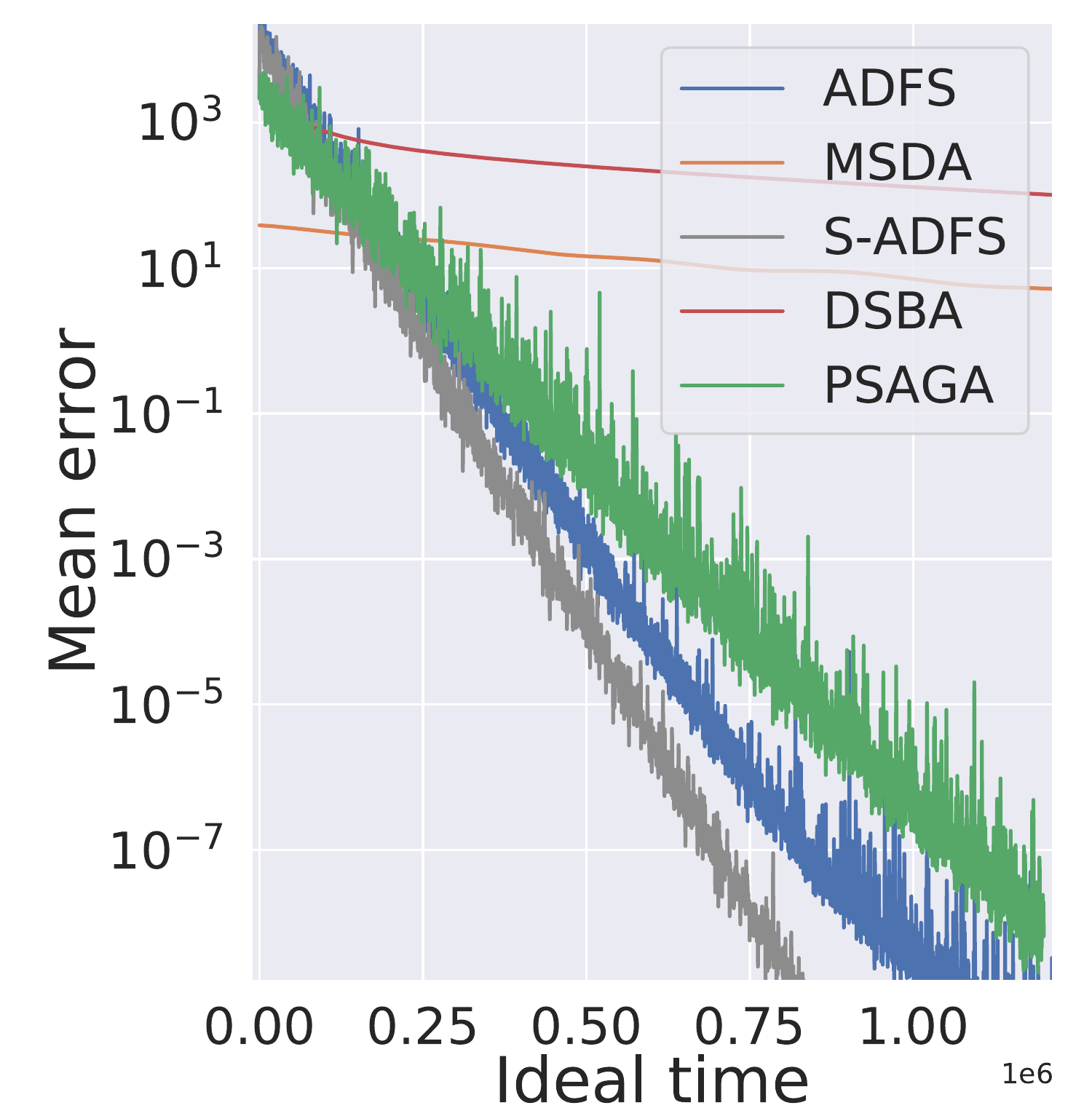}
  \vspace{-15pt}
\caption{Higgs, $n=4$}
\label{fig:2x2_tau1_m1000}
\end{subfigure}%
\begin{subfigure}{.32\textwidth}
  \centering
  \includegraphics[width=1.0\linewidth]{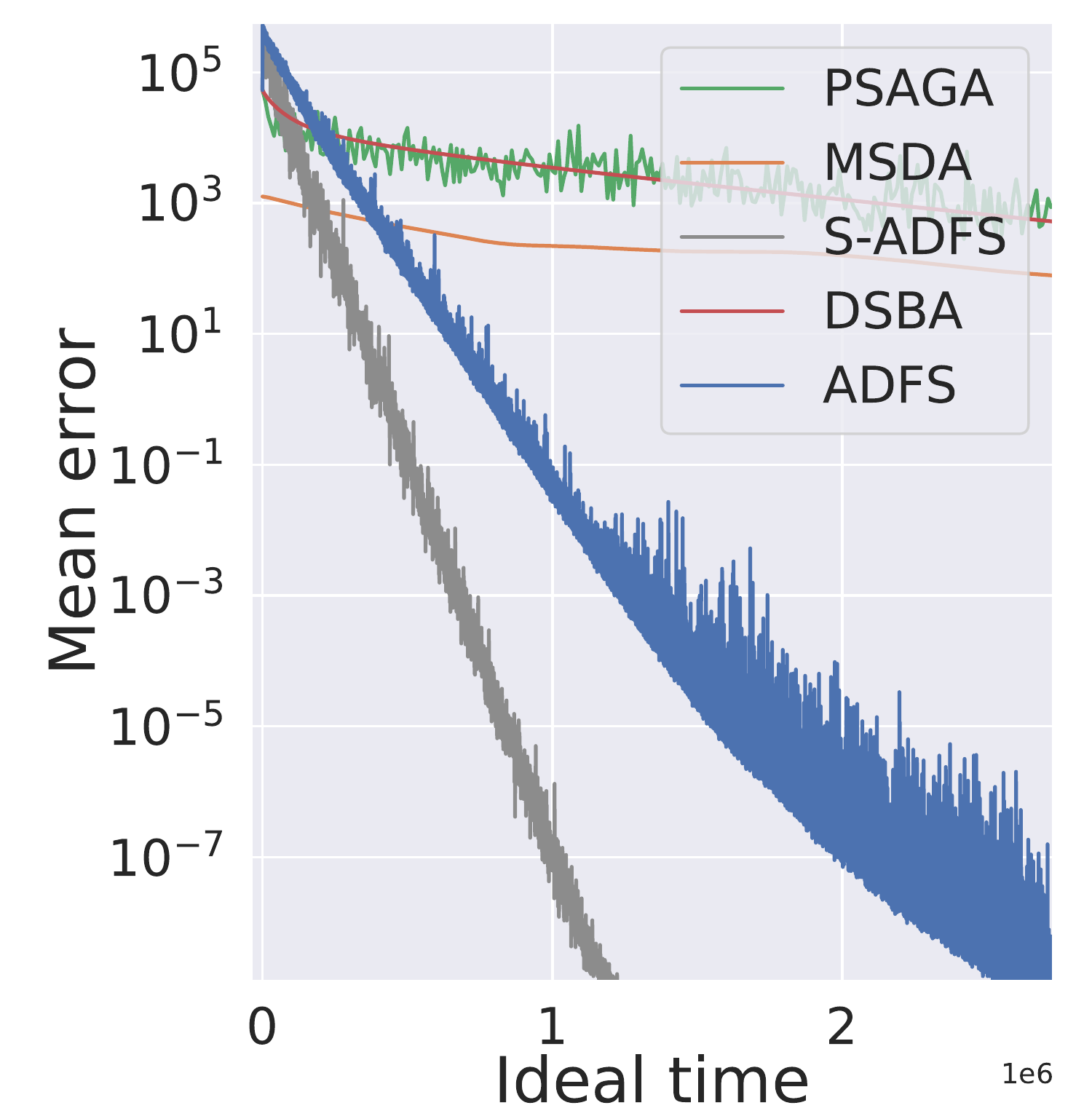}
  \vspace{-15pt}
  \caption{Higgs, $n=100$}
\label{fig:10x10_m300}
\end{subfigure}
\begin{subfigure}{.32\textwidth}
  \centering
  \includegraphics[width=1.0\linewidth]{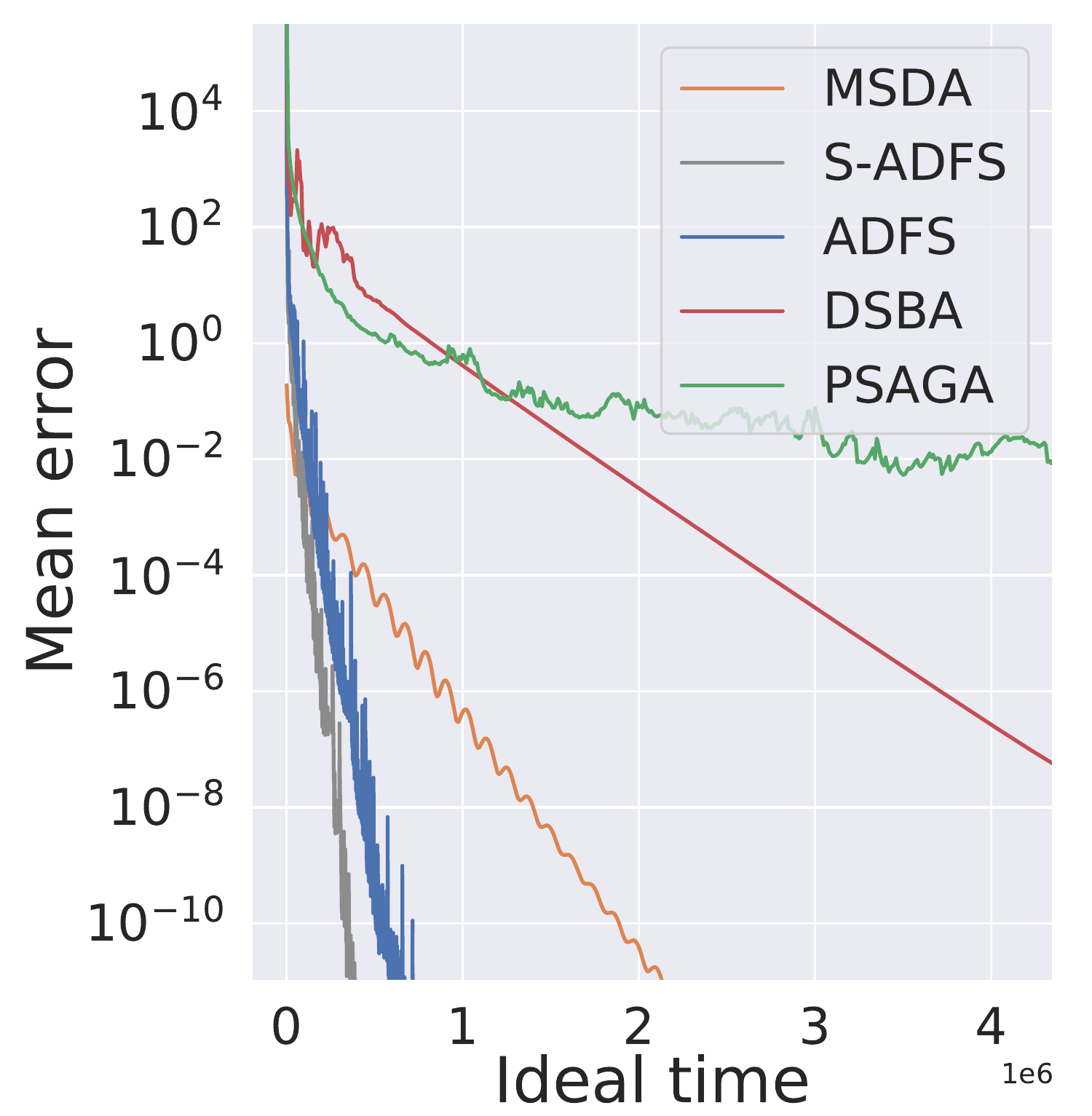}
  \vspace{-15pt}
  \caption{Covtype, $n=100$}
\label{fig:7x7_tau5_m10000}
\end{subfigure}
\vspace{5pt}
\caption{Performances of various decentralized algorithms on the logistic regression task with $m = 10^4$ points per node, regularization parameter $\sigma=1$ and communication delays $\tau=5$ on 2D grid networks of different sizes.}
\label{fig:test}
\end{figure*}
In this section, we illustrate the theoretical results by showing how ADFS compares with MSDA~\citep{scaman2017optimal}, Point-SAGA~\citep{defazio2016simple}, and DSBA~\citep{shen2018towards}. We also compare the synchronous version of ADFS (S-ADFS) to the locally synchronous one of the conference paper (ADFS)~\cite{hendrikx2019accelerated}. All algorithms (except for DSBA, for which we fine-tuned the step-size) were run with out-of-the-box hyperparameters given by theory on data extracted from the standard Higgs and Covtype datasets from LibSVM. The underlying graph is a 2D grid network. Experiments were run in a distributed manner on an actual computing cluster. Yet, plots are shown for \emph{idealized times} in order to abstract implementation details as well as ensure that reported timings were not impacted by the cluster status. All the details of the experimental setup can be found in Appendix~\ref{app:experimental_setting}. An implementation of S-ADFS is also available in supplementary material.

First of all, we note on all the plots from Figure~\ref{fig:test} that ADFS and S-ADFS exhibit very similar performances. S-ADFS is always slightly faster because it suffers from no waiting time but, as argued in the conference paper~\cite{hendrikx2019accelerated}, the waiting time due to the local synchrony of ADFS is rather small. Although S-ADFS is easier to implement (series of synchronous rounds), it offers less flexibility than ADFS to deal with identified stragglers. In the next paragraph, we refer to both S-ADFS and ADFS as ADFS since the differences are rather small. 

Figure~\ref{fig:2x2_tau1_m1000} shows that, as predicted by theory, ADFS and Point-SAGA have similar rates on small networks. In this case, ADFS uses more computing power but has a small overhead. Figures~\ref{fig:10x10_m300} and~\ref{fig:7x7_tau5_m10000} use a much larger grid to evaluate how these algorithms scale. In this setting, Point-SAGA is the slowest algorithm since it has 100 times less computing power available. MSDA performs quite well on the Covtype dataset thanks to its very good network scaling. Yet, the $m\sqrt{\kappa}$ factor in its rate makes it scales poorly with the condition number $\kappa$, which explains why it struggles on the Higgs dataset. DSBA is slow as well despite the fine-tuning because it is the only non-accelerated method, and it has to communicate after each proximal step, thus having to wait for a time $\tau=5$ at each step. ADFS does not suffer from any of these drawbacks and therefore outperforms other approaches by a large margin on these experiments. This illustrates the fact that ADFS combines the strengths of accelerated stochastic algorithms, such as Point-SAGA, and fast decentralized algorithms, such as MSDA. 

\section{Conclusion}
In this paper, we develop an algorithmic framework for accelerated decentralized stochastic optimization based on accelerated block coordinate descent with arbitrary sampling. It is an extension of the conference paper~\citep{hendrikx2019accelerated} that provides stronger convergence results, and allows more flexibility in the algorithm design. This flexibility is obtained thanks to the arbitrary block sampling, so it is possible to transparently use global synchronous communications as well as local pairwise communications, or anything in between. The rate of ADFS explicitly mixes optimization-related and graph-related quantities, so that is is possible to adapt the parameters of the algorithm to heterogeneous problems with specific structure, as done with the line graph for instance.  

We also provide a lower bound for decentralized stochastic optimization, and show that a synchronous implementation of ADFS almost matches this lower bound when parameters are chosen in a suitable way. The bound is exactly matched for generalized linear models, and otherwise a small gap due to the difference between the stochastic and batch condition numbers may exist. The problem of closing this gap in full generality remains open. 

\section*{Acknowledgement}
We acknowledge support from the European Research Council (grant SEQUOIA 724063). 

\begin{table}
    \centering
    \begin{tabular}{c|c}
    \hline
         $\tau$ & Communication time \\
         \hline
         $n$ & Number of computing nodes \\
         \hline
         $m$ & Number of local functions per node \\
         \hline
         $\sigma_i$ & Strong convexity of function $f_i$\\
         \hline
         $L_{i,j}$ & Smoothness of function $f_{i,j}$\\
         \hline
         $\kappa_i$ & $1 + \sigma_i^{-1} \sum_{j=1}^m L_{i,j}$ \\
         \hline
         $\kappa_s$ & $\max_i \kappa_i$ \\
         \hline
         $M_i$ & smoothness of Function $f_i$\\
         \hline
         $\kappa_b$ & $M_i / \sigma_i$ \\
         \hline
         $L$ & Laplacian of the communication network $f_i$\\
         \hline
         $\lambda_{\min}^+(L)$ & Smallest non-zero eigenvalue of $L$\\
         \hline
         $\gamma$ & $\lambda_{\min}^+(L) / \lambda_{\max}(L)$\\
         \hline
         $R_{k,\ell}$ & $e_{k,\ell}^\top A^\dagger A e_{k,\ell}$\\
         \hline
    \end{tabular}
    \caption{Notation table}
    \label{tab:notation_recap}
\end{table}

\bibliographystyle{siamplain}
\bibliography{journal_biblio}

\newpage

\appendix 

Section~\ref{app:lower_bound} presents the proofs for the lower bounds. Then, Section~\ref{app:apcg} presents APCG with arbitrary block sampling as well as efficient formulations. Section~\ref{app:algo_perfs} presents the derivations required to obtain ADFS from the extended APCG algorithm as well as the analysis of the speed of ADFS for a specific choice of parameters. Section~\ref{app:adfs_efficient} presents formulations of ADFS that can be implemented efficiently, and Section~\ref{app:experimental_setting} details the experimental setting.

\section{Lower bounds proofs}
\label{app:lower_bound}
The goal of this Section is to prove the various lower bounds presented in Section~\ref{sec:lower_bound}. Proofs are based on the work of~\citet{scaman2017optimal}, with separable local functions as in~\citet{lan2017optimal} to take into account the finite sum aspect. For simplicity, the proofs are presented for $x\in \ell_2$, the space of sequences with summable squares, but they can be adapted to $x \in \R^d$ as done in~\citet{lan2017optimal}. We prove in this section Theorem~\ref{thm:lb_general}.

\begin{proof}
We consider $Q$ a set of nodes and $Q_\Delta^c$ the set of nodes at distance at least $\Delta$ from $Q$ in the graph $G$. Let $L, \sigma > 0$ be such that $L \geq \sigma$ and $\kappa_\ell > 3L / \sigma$. Then, we define for $y \in \ell_2$ functions $\psi_i^Q$ such that:
\begin{equation}
    f_i^Q(y) = \frac{1}{2|Q|} \left[\frac{\sigma}{3}\|y\|^2 + \frac{L - \sigma}{4}(y^\top  M_1 y - e_1^\top  y)\right] \hbox{ if } i \in Q,
\end{equation}
\begin{equation}
    f_i^Q(y) = \frac{1}{2|Q_\Delta^c|} \left[\frac{\sigma}{3}\|y\|^2 + \frac{L - \sigma}{4}y^\top  M_2 y\right] \hbox{ if } i \in Q_d^c,
\end{equation}
\begin{equation}
    f_i^Q(y) = \frac{\sigma}{6(n - |Q_\Delta^c| - |Q|)} \|y\|^2 \hbox{ otherwise,}
\end{equation}
where $M_1$ is the infinite block diagonal matrices with $\begin{pmatrix} 1 & - 1 \\ - 1 & 1 \end{pmatrix}$ on the diagonal and $M_2 = \begin{pmatrix} 1 & 0 \\ 0 & M_1 \end{pmatrix}$. We then define for all $i \in \{1, \cdots, n\}$:
$$\psi_i: x \in \ell_2^m \mapsto \sum_{j=1}^m f_i^Q(x_j) \hbox{, } f = \sum_{i=1}^n f_i^Q, \hbox{ and } \psi = \sum_{i=1}^n \psi_i.$$
Note that we have that for all $i,j$, $\kappa_\ell \geq 3L / \sigma$ since $0 \preceq M_1 + M_2 \preceq 4 I$. Besides, the solution of $\min_{y \in \ell_2} f(y)$ is $y^*$ such that for $k \geq 1$, the $k$-th coordinate of $y^*$ is $y^*(k) = q^k$ where $q = \frac{\sqrt{L / \sigma} - 1}{\sqrt{L / \sigma} + 1}$. Indeed, $y^*$ is such that for all $k \geq 1$,
$$\sigma y^*(k) + \frac{L - \sigma}{4}\left[2y^*(k) - y^*(k-1) - y^*(k+1)\right] = 0,$$
where $y^*(0) = 1$ by convention. We now consider a sequence $x^t \in \ell_2^{n \times m}$ generated by a black-box optimization procedure as defined in Section~\ref{sec:black_box_model} and such that $x^0 = 0$ without loss of generality (it comes down to optimizing a shifted version of $\psi$). Therefore, $x^t_{i,j} \in \ell_2$ corresponds to the $j$-th entry of the local parameter of node $i$, and we define similarly $x^* \in \ell_2^{n \times m}$ such that $x^*_{i,j} = y^*$. We then write: 
$$\sum_{i=1}^n \sum_{j=1}^m \|x_{i,j}^t - x^*_{i,j}\|^2 \geq \sum_{i=1}^n \sum_{j=1}^m \sum_{l \geq k_{i,j}(t)}\|x^*_{i,j}(l)\|^2 \geq \sum_{i=1}^n \sum_{j=1}^m \frac{1}{1 - q^2} q^{2k_{j}(t)},$$
where $k_{j}(t)$ is the first index such that $x_{i,j}^t(l) = 0$ for all $i$ and $l \geq k_j(t)$. Using the fact that $\|x_{i,j}^0 - x^*_{i,j}\|^2 = (1 - q)^{-1}$ for all $i$, we write:
$$\esp{\sum_{i=1}^n\sum_{j=1}^m \frac{\|x_{i,j}^t - x^*_{i,j}\|^2}{\|x_{i,j}^0 - x^*_{i,j}\|^2}} \geq \frac{1 - q}{1 - q^2} n \sum_{j=1}^m \esp{q^{2k_{j}(t)}}.$$
An upper bound on $k_j(t)$ thus gives a lower bound on the expected error. We now consider the following two extreme cases: 

\textbf{Case 1 : Communication bottleneck.} The first one consists in considering that computations are instant nodes in $Q$ and $Q_\Delta^c$ perform a gradient update using function $i$ as soon as they receive the value. This is an extremely favorable case in which communication is the only bottleneck. Let us denote ${\rm pos}_0$ the operator that gives the position of the last non-zero coordinate of a sequence in $\ell_2$. In particular, $\max_i \poso(x^t_{i,j}) = k_j(t)$. At the beginning, $x_0 = 0$ so $k_j(0) = 1$. Due to the structure of $f$, if $\poso(x)$ is odd then $\poso(\nabla f_i^Q(x)) = \poso(x) + \mathds{1}\{i \in Q\}$. If $\poso(x)$ is even then $\poso(\nabla f_i^Q(x)) = \poso(x) + \mathds{1}\{i \in Q_\Delta^c\}$. Similarly, if we denote $p_x = \prox_{\eta f_i^Q}(x)$ then $p_x + \eta \nabla f_i^Q(p_x) = x$ so the same reasoning can be applied since $p_x$ is unique and if $\poso(x)$ is odd, $p_x(k) = 0$ is a solution for $k > \poso(x) + \mathds{1}\{i \in Q\}$.

Therefore, nodes in $Q$ can only increase $k_j(t)$ if it is odd, and node in $Q^c_\Delta$ can only increase $k_j(t)$ if it is even. Considering that a message takes time at least $\Delta \tau$ (with $\Delta$ the diameter of the network) to go from $Q$ to $Q_\Delta^c$ we get:
\begin{equation} \label{eq:kt_comm}
k_{j}(t) \leq 1 + \frac{t}{\Delta \tau}, \hbox{ and so }  \esp{q^{2k_{j}(t)}} \geq q^{2 + \frac{2t}{\Delta \tau}}.
\end{equation}
where the $2$ term is here to account for the fact that no communication is needed for the first step. This bound corresponds to the case in~\cite{scaman2017optimal} in which computation of the useful gradient starts as soon as possible.\\

\textbf{Case 2 : Computations bottleneck.} The other bound is obtained by considering the other extreme case, in which communications are instantaneous. At a given time $t$, due to the form of the local functions, the only nodes that can improve the error for a given dimension are either the ones in $Q$ or the ones in $Q_c^\Delta$. Consider it is a node in $Q$, then the message needs to be sent to a node in $Q_c^\Delta$ and from there, only nodes in $Q_c^\Delta$ will be able to increase $k_{j}(t)$. Therefore, if we neglect both the communication time and the time it takes for nodes in $Q_c^\Delta$ to increase $k_j(t)$, we obtain that $k_j(t)$ is bounded by two times the number of time progress has been made on coordinate $j$. Since node $i$ can only compute first-order characteristics for the function $\psi_{i,\zeta_i(t)} :x \mapsto f_i^Q(x_{\zeta_i(t)})$, this leads to:
$$k_j(t) \leq 2 \sum_{l=1}^t \sum_{i \in Q} \mathds{1}\{\zeta_i(t) = j\},$$
since each evaluation takes time $1$ and only nodes in $Q$ can increase even dimensions. In particular, we have that
$$\sum_{j=1}^m k_j(t) \leq 2 \sum_{l=1}^t \sum_{i \in Q} \sum_{j=1}^m \mathds{1}\{\zeta_i(t) = j\} \leq 2 |Q| \lceil t \rceil.$$
We can then use Jensen inequality with the convex function $f: x \mapsto q^{2x}$ and the fact that $q \leq 1$ to write that:
\begin{equation} \label{eq:kt_comp}
  \frac{1}{m}\sum_{i=j}^m q^{2k_j(t)} \geq q^{\frac{2}{m}\sum_{j=1}^m k_j(t)} \geq q^{\frac{4 |Q| \lceil t \rceil}{m}}.
\end{equation}
Therefore, parallelism is very limited in this case because the bound only nodes in $Q$ actually contribute to the progress. This is actually the case as well in~\citet{scaman2017optimal}.
In the end, we can lower bound the error by the max of Equations~\eqref{eq:kt_comm} and~\ref{eq:kt_comp}. In order to have a simpler expression, we lower bound the maximum of the two terms by their average to obtain:
$$ \frac{2}{nm} \frac{1 - q^2}{1 - q} \esp{\sum_{i=1}^n \sum_{j=1}^m\frac{\|x_{i,j}^t - x^*_{i,j}\|^2}{\|x_{i,j}^0 - x^*_{i,j}\|^2}} \geq q^{\frac{4 |Q| \lceil t \rceil}{m}} + q^{2 + \frac{2\lceil t\rceil}{\Delta\tau}}.$$
Since $\|x^t - x^*\|^2 = \sum_{i=1}^n \sum{j=1}^m\|x_{i,j}^t - x_{i,j}^*\|^2$ and for all $i,i^\prime, j, j^\prime,$ we have $x_{i,j}^0 = x_{i^\prime, j^\prime}$ and $x^*_{i,j} = x_{i^\prime, j^\prime}^*$ then $\|x^0 - x^*\|^2 = nm \|x_{i,j}^0 - x^*\|^2$ for some $i,j$. In particular, 
$$ 2 \frac{1 - q^2}{1 - q} \esp{\frac{\|x^t - x^*\|^2}{\|x^0 - x^*\|^2}} \geq q^{\frac{4 |Q| \lceil t \rceil}{m}} + q^{2 + \frac{2\lceil t\rceil}{\Delta\tau}}.$$
Then, we write $q = \frac{\sqrt{L / \sigma} - 1}{\sqrt{L / \sigma} + 1} = 1 - \frac{2}{\sqrt{L / \sigma} + 1}$ and $1 + \sqrt{L / \sigma} = (m + m\sqrt{L / \sigma}) / m = (m + \sqrt{m \kappa_s / 3}) / m$ since the objective is separable in $j$. This yields:
\begin{align*}
  2 \frac{1 - q^2}{1 - q} \esp{\frac{\|x^t - x^*\|^2}{\|x^0 - x^*\|^2}} \geq 
  \left(1 - \frac{2m}{m + \sqrt{m\kappa_s / 3})}\right)^{\frac{4|Q|\lceil t \rceil}{m}} + \left(1 - \frac{2}{1 + \sqrt{\kappa_\ell / 3}}\right)^{2 + \frac{2\lceil t\rceil}{\Delta\tau}}.
\end{align*}
It is then possible to pick $\Delta$ as the diameter of the graph and $Q = \{u\}$ where $u \in \arg\max_{v} d(u,v)$ where $d(u,v)$ is the distance between nodes $u$ and $v$ in the graph $\mathcal{G}$.
Note that the way of choosing $\zeta_i(t)$ (\emph{e.g.}, deterministically or randomly) does not matter. 
\end{proof}

\section{Generalized APCG}
\label{app:apcg}
\subsection{Efficient implementation}
\label{app:efficient_apcg}
This section presents efficient implementations of the generalized APCG algorithm. The main goal is to avoid as much as possible to perform convex combinations of dense vectors. The main changes in Algorithm~\ref{algo:generalized_apcg_eff_sc} are that we express the proximal operator of line 5 in a slightly different but equivalent form and that line $6$ requires a matrix product to take into account the block aspect and the strong convexity in an arbitrary norm. The proof is a straightforward adaptation of~\citet{lin2015accelerated}. Note that the full vector $w_t$ never actually needs to be formed so local updates are sparse.

\begin{algorithm}
\caption{Efficient Generalized APCG$(\rho, \sigma_A, p_i)$, Strongly Convex Case.}
\label{algo:generalized_apcg_eff_sc}
\begin{algorithmic}[1]
\STATE $u_0 = 0$, $z_0 = 0$, $\phi = \frac{1 - \rho}{1 + \rho}$, $\eta = \frac{\rho}{\sigma_A}$
\WHILE{$t < T$}
\STATE $w_t = - \phi^{t+1}u_t + v_t$
\STATE $g_t = \eta P_b^\dagger \nabla q_A(\phi^{t+1}u_t + z_t)$
\STATE $h_t^{(i)} = {\rm prox}_{\eta p_i^{-1} \psi_i}\left(w_{t}^{(i)} - g_t^{(i)}\right) - w_t^{(i)}$ for all $i \in b$, $0$ otherwise.
\STATE $u_{t+1} = u_t - \frac{I - \rho P_{b_t}^\dagger A^\dagger A}{2\phi^{t+1}}h_t, \ \ \ z_{t+1} = z_t + \frac{I + \rho P_{b_t}^\dagger A^\dagger A}{2}h_t$
\ENDWHILE
\RETURN $\phi^{T}u_T + z_T$
\end{algorithmic}
\end{algorithm}

We now present the convex case, which is an adaptation of~\citet{fercoq2015accelerated}. We therefore refer the interested reader to this paper for the details of the equivalence between Algorithm~\ref{algo:generalized_apcg} in the convex case and Algorithm~\ref{algo:generalized_apcg_eff_c}.

\begin{algorithm}
\caption{Efficient Generalized APCG$(\rho, \sigma_A, p_i)$, Convex Case.}
\label{algo:generalized_apcg_eff_c}
\begin{algorithmic}[1]
\STATE $u_0 = 0$, $v_0 = 0$, $\alpha_0 = \min_{i, \psi_i \neq 0} p_i$, $\eta_t = \frac{1}{\alpha_t S^2}$
\WHILE{$t < T$}
\STATE $g_t = \eta_t P_b^\dagger \nabla q_A(\alpha_t^2 u_t + z_t)$
\STATE $z_{t+1}^{(i)} = {\rm prox}_{\eta_t p_i^{-1} \psi_i}\left(z_t^{(i)} - g_t^{(i)}\right)$ for all $i \in b$, $0$ otherwise.
\STATE $u_{t+1} = u_t - \frac{I - \alpha_t P_{b_t}^\dagger A^\dagger A}{2\alpha_t^2}(z_{t+1} - z_t)$
\STATE $\alpha_{t+1} = \frac{\sqrt{\alpha_t^4 + 4\alpha_t^2} - \alpha_t^2}{2} = \frac{2}{1 + \sqrt{1 + 4 \alpha_t^{-2}}}$
\ENDWHILE
\RETURN $\alpha_{T - 1}^2 u_{T} + v_{T}$
\end{algorithmic}
\end{algorithm}

\subsection{Proof of Theorem~\ref{thm:gen_apcg}}
Before starting the proof, we define $w_t = (1 - \beta_t) v_t + \beta_t y_t$, and for $v \in \R$:
$$V_i^t(v) = \frac{B_{t+1}p_i}{2 a_{t+1}} \|v - w_t^{(i)} + \eta_i e_i^\top  \nabla f(y_t)\|^2 + \psi_i(v).$$ Then, we give the following lemma, which generalizes the proofs in~\citet{lin2015accelerated} and~\citet{fercoq2015accelerated} by considering non-uniform probabilities and that works with blocks of coordinates for both the convex and the strongly convex cases. The proof is given later.

\begin{lemma}
\label{lemma:lyapunov_psi}
If either $1 - \beta_t - \frac{\alpha_t}{p_i} \geq 0$ or $\alpha_t = \beta_t$ and $1 - \frac{\alpha_t}{p_i} \geq 0$ for any $i$ such that $\psi_i \neq 0$, then for any $t$ and $i$ such that $\psi_i \neq 0$, we can write $x_t^{(i)} = \sum_{l=0}^t \delta^{(i)}_t(l) v_l^{(i)}$ such that $\sum_{l=0}^t \delta^{(i)}_t(l) = 1$ and for any $l$, $\delta^{(i)}_t(l) \geq 0$. We define $\hat{\psi}_t^{(i)} = \sum_{l=0}^t \delta^{(i)}_t(l) \psi_i(v_l^{(i)})$ and $\hat{\psi}_t = \sum_{i=1}^d \hat{\psi}_t^{(i)}$. Then, if $R_i = 1$ whenever $\psi_i \neq 0$, $\psi(x_t) \leq \hat{\psi}_t$ and:
\begin{equation}
    \mathbb{E}_{i_t}\left[\hat{\psi}_{t+1}\right] \leq \alpha_t \psi(\tilde{v}_{t+1}) + (1 - \alpha_t)\hat{\psi}_t.
\end{equation}
where $\tilde{v}_{t+1}^{(i)} = \arg \min_v V_i^t(v)$ for all $i$. In particular, $v_{t+1}^{(i)} = \tilde{v}_{t+1}^{(i)}$ if $i \in b_t$ and $v_{t+1}^{(i)} = w_t^{(i)}$ if $i \notin b_t$.
\end{lemma}

Note that Lemma~\ref{lemma:lyapunov_psi} is a small generalization to arbitrary sampling probabilities of the beginning of the proof in~\citep{lin2015accelerated}. We now introduce and Lemma~\ref{lemma:master_inequality}, which is the main inequality from which the rest of the proof follows directly.

\begin{lemma}
\label{lemma:master_inequality}
For any block of coordinates $b$, the following inequality holds:
\begin{align}
\label{eq:sc_prox}
\begin{split}
    \frac{1}{2\eta_t} [&\|v_{t+1} - \theta^\star\|^2_{A^\dagger A} + \|v_{t+1} - w_t\|^2_{A^\dagger A} - \|\theta^\star - w_t\|^2_{A^\dagger A}] \\
    &\leq \langle P_b^\dagger \nabla q_A(y_t), \theta^\star - v_{t+1} \rangle_{A^\dagger A} + \sum_{i \in b}\frac{1}{p_i} \left[\psi_i\left({\theta^\star}^{(i)}\right) - \psi_i\left(v_{t+1}^{(i)}\right)\right].
    \end{split}
\end{align}
\end{lemma}

\begin{proof}
We note $v_{t+1}^\perp$ the restriction of $v_{t+1}$ to coordinates $i$ such that $\psi_i = 0$. Similarly, we note $b^\perp$ the restriction of the block $b$ to coordinates $i$ such that $\psi_i = 0$. With these notations, we write:
\begin{align}
\label{eq:sc_prox_perp}
\begin{split}
    \frac{1}{2\eta_t} [\|v_{t+1}^{\perp} - & {\theta^\star}^{\perp}\|^2_{A^\dagger A} + \|v_{t+1}^{\perp} - w_t^{\perp}\|^2_{A^\dagger A} - \|{\theta^\star}^{\perp} - w_t^{\perp}\|^2_{A^\dagger A}] \\
    &\leq \sum_{i \in b^\perp}\frac{1}{p_i}\left[\langle \nabla_i q_A(y_t), \theta^\star - v_{t+1} \rangle_{A^\dagger A} + \psi_i\left({\theta^\star}^{(i)}\right) - \psi_i\left(v_{t+1}^{(i)}\right)\right].
    \end{split}
\end{align}
Equation~\eqref{eq:sc_prox_perp} follows directly from using $v_{t+1}^\perp = w_t^\perp - \sum_{i \in b^\perp} \frac{\eta_t}{p_i}\nabla_i q_A(y_t)$ and basic algebra (expanding the squared terms).

If $i \in b$ is such that $\psi_i \neq 0$, we use the strong convexity of $V_i^t$ at points $v_{t+1}^{(i)}$ (its minimizer, by definition) and ${\theta^\star}^{(i)}$ ($i$-th coordinate of a minimizer of $F$) to write that $V_i^t(v_{t+1}^{(i)}) + \frac{B_{t+1}p_i}{2a_{t+1}}\|v_{t+1}^{(i)} - {\theta^\star}^{(i)}\|^2 \leq V_i^t({\theta^\star}^{(i)})$. This is a key step from the proof of~\citet{lin2015accelerated} and uses the same arguments as Lemma~3 from~\citet{fercoq2015accelerated}. Then, expanding the $V_i^t$ terms yields:
\begin{align*}
     &\|v_{t+1}^{(i)} - {\theta^\star}^{(i)} \|^2 + \|v_{t+1}^{(i)} - w_t^{(i)} + \frac{a_{t+1}}{B_{t+1} p_i}\nabla_i q_A(y_t)\|^2\\
      &- \|{\theta^\star}^{(i)} - w_t^{(i)} + \frac{a_{t+1}}{B_{t+1} p_i}\nabla_i q_A(y_t)\|^2 \leq  \frac{2a_{t+1}}{B_{t+1}p_i}\left[\psi_i\left({\theta^\star}^{(i)}\right) - \psi_i\left(v_{t+1}^{(i)}\right)\right].
\end{align*}
If we pull gradient terms out of the squares this yields:
\begin{align}
\label{eq:sc_prox_i}
\begin{split}
    \frac{1}{2\eta_t} [\|v_{t+1}^{(i)} - &{\theta^\star}^{(i)}\|^2_{A^\dagger A} + \|v_{t+1}^{(i)} - w_t^{(i)}\|^2_{A^\dagger A} - \|{\theta^\star}^{(i)} - w_t^{(i)}\|^2_{A^\dagger A}] \\
    &\leq \frac{1}{p_i}\left[\langle \nabla_i q_A(y_t), \theta^\star - v_{t+1} \rangle_{A^\dagger A} + \psi_i\left({\theta^\star}^{(i)}\right) - \psi_i\left(v_{t+1}^{(i)}\right)\right].
    \end{split}
\end{align}
Finally, if $i \notin b$ is such that $\psi_i \neq 0$ then $v_{t+1}^{(i)} = w_t^{(i)}$ and so 
$$\|v_{t+1}^{(i)} - {\theta^\star}^{(i)}\|^2_{A^\dagger A} + \|v_{t+1}^{(i)} - w_t^{(i)}\|^2_{A^\dagger A} - \|{\theta^\star}^{(i)} - w_t^{(i)}\|^2_{A^\dagger A} = 0.$$
Note that $A^\dagger Ae_i = e_i$ for all $i$ such that $\psi_i \neq 0$ since $e_i^\top A^\dagger Ae_i = 1$ and $A^\dagger A$ is a projector. Therefore,
\begin{equation}
\label{eq:vt1proj}
    \|v_{t+1} - \theta^\star\|^2_{A^\dagger A} = \|v_{t+1}^\perp - {\theta^\star}^\perp\|^2_{A^\dagger A} + \sum_{i, \psi_i \neq 0} \|v_{t+1}^{(i)} - {\theta^\star}^{(i)}\|^2_{A^\dagger A}
\end{equation}
so we can sum Equation~\eqref{eq:sc_prox_perp} with Equation~\eqref{eq:sc_prox_i} for all $i$ to finish the proof.
\end{proof}

\begin{proof}[Proof of Theorem~\ref{thm:gen_apcg}]
This proof follows the same general structure as Nesterov and Stich ~\citep{nesterov2017efficiency}. In particular, it follows from expanding the $\|v_{t+1} - \theta^\star\|^2$ term. In the original proof, $v_{t+1} = w_t - g$ where $g$ is a gradient term so the expansion is rather straightforward. In our case, $v_{t+1}$ is defined by a proximal mapping so a bit more work is required. Yet, similar terms appear, along with the function values of the non-smooth term that we control with Lemma~\ref{lemma:lyapunov_psi}. This expansion is done by Lemma~\ref{lemma:master_inequality}, which relies on using the strong convexity of the proximal mapping. 

We now evaluate each term of Equation~\eqref{eq:sc_prox}. First of all, we use that $y_t - x_{t+1} = \alpha_t P_b^\dagger A^\dagger A (w_t - v_{t+1})$ to write:
\begin{align*}
    &\esp{\langle P_b^\dagger \nabla q_A(y_t), \theta^\star - v_{t+1} \rangle_{A^\dagger A}} \\
    & =\esp{\langle P_b^\dagger\nabla q_A(y_t), \theta^\star - w_t\rangle_{A^\dagger A}} + \esp{\nabla q_A(y_t)^\top  P_b^\dagger A^\dagger A (w_t - v_{t+1})} \\
    &= \langle \nabla q_A(y_t), \theta^\star - w_t \rangle_{A^\dagger A} + \alpha_t^{-1} \esp{\nabla q_A(y_t)^\top  (y_t - x_{t+1})}.
\end{align*}
The rest of this proof closely follows the analysis from Hendrikx et al.~\citep{hendrikx2018accelerated}, which is an adaptation of Nesterov and Stich~\citep{nesterov2017efficiency} to strong convexity on a subspace. The main difference is that it is also necessary to control the function values of $\psi$, which is done using Lemma~\ref{lemma:lyapunov_psi}. For the first term, we use the strong convexity of $f$ as well as the fact that $w_t = y_t - \frac{1 - \alpha_t}{\alpha_t} (x_t - y_t)$ to obtain:
\begin{align*}
a_{t+1}& \nabla q_A(y_t)^\top  A^\dagger A (\theta^\star - w_t) = a_{t+1} \nabla q_A(y_t)^\top  A^\dagger A \left(\theta^\star - y_t + \frac{1 - \alpha_t}{\alpha_t}(x_t - y_t)\right)\\
&\leq a_{t+1} \left(q_A(\theta^\star) - q_A(y_t) - \frac{1}{2} \sigma_A \|y_t - \theta^\star\|^2_{A^\dagger A} + \frac{1 - \alpha_t}{\alpha_t}(q_A(x_t) - q_A(y_t)) \right) \\
&\leq a_{t+1} q_A(\theta^\star) - A_{t+1} q_A(y_t) + A_t q_A(x_t) - \frac{1}{2} a_{t+1} \sigma_A \|y_t - \theta^\star\|^2_{A^\dagger A}.
\end{align*}
For the second term we use the smoothness of $q_A$ and then the fact that $x_{t+1} - y_t$ has support on $U_k$ only (just like $v_{t+1} - w_t$), as well as the fact that $A^\dagger A$ is symmetric to obtain:
\begin{align*}
    &\frac{a_{t+1}}{\alpha_t}\nabla q_A(y_t)^\top (y_t - x_{t+1})\\
    &\leq A_{t+1}\left[q_A(y_t) - q_A(x_{t+1})\right] + \frac{a_{t+1}}{2\alpha_t}\|x_{t+1} - y_t\|^2_M \\
    &\leq A_{t+1}\left[q_A(y_t) - q_A(x_{t+1})\right] + \frac{a_{t+1}\alpha_t}{2}\|A^\dagger A(v_{t+1} - w_t)\|^2_{P_b^\dagger M P_b^\dagger}\\
    &\leq A_{t+1}\left[q_A(y_t) - q_A(x_{t+1})\right] +\frac{a_{t+1}^2 \lambda_{\max}((A^\dagger A P_b^\dagger M P_b^\dagger A^\dagger A)}{2A_{t+1}} \|v_{t+1} - w_t\|^2_{A^\dagger A}.
\end{align*}
Noting $\Delta q_A(x_t) = \esp{q_A(x_t)} - q_A(\theta^\star)$ and remarking that $a_{t+1} = A_{t+1} - A_t$, we obtain, using that $\alpha_t = \frac{a_{t+1}}{A_{t+1}}$: 
\begin{align*}
    &\esp{a_{t+1}\langle P_b^\dagger \nabla q_A(y_t), \theta^\star - v_{t+1} \rangle_{A^\dagger A}}\\
    &\leq A_t \Delta q_A(x_t) - A_{t+1} \Delta q_A(x_{t+1})+ \frac{B_{t+1}}{2}\esp{ \|w_t - v_{t+1}\|^2_{A^\dagger A}}- \frac{a_{t+1}\sigma_A}{2}\|y_t - \theta^\star\|^2_{A^\dagger A}.
\end{align*}
Using Lemma~\ref{lemma:lyapunov_psi}, we derive in the same way:
\begin{align*}
    &\esp{\frac{a_{t+1}}{p_i} \left[\psi_i\left({\theta^\star}^{(i)}\right) - \psi_i\left(v_{t+1}^{(i)}\right)\right]}\\
     &= a_{t+1}\psi(\theta^\star) - A_{t+1} \alpha_t  \psi(\tilde{v}_{t+1})\\
    &\leq A_t \left(\esp{\hat{\psi}_t} - \psi(\theta^\star)\right) - A_{t+1}\left(\esp{\hat{\psi}_{t+1}} - \psi(\theta^\star) \right).
\end{align*}
Now, we can multiply Equation~\eqref{eq:sc_prox} by $\frac{a_{t+1}}{p_i}$ and take the expectation over $i$. The $\|v_{t+1} - w_t\|^2_{A^\dagger A}$ terms cancel and we obtain:
\begin{align*}
\frac{B_{t+1}}{2}\esp{\|v_{t+1} - \theta^\star\|^2_{A^\dagger A}}& + A_{t+1} \Delta \hat{F}_A(x_{t+1})\\
&\leq A_t \Delta \hat{F}_A(x_t) + \frac{B_{t+1}}{2}\|w_t - \theta^\star\|^2_{A^\dagger A} - \frac{a_{t+1}\sigma_A}{2}\|y_t - \theta^\star\|^2_{A^\dagger A},
\end{align*}
where $\Delta \hat{F}_A(x_t) = \Delta q_A(x_t) +  \esp{\hat{\psi}_t} - \psi(\theta^\star)$. Convexity of the squared norm yields $\|w_t - \theta^\star\|^2_{A^\dagger A} \leq (1 - \beta_t)\|v_t - \theta^\star\|^2_{A^\dagger A} + \beta_t \|y_t - \theta^\star\|^2_{A^\dagger A}$. Now remarking that $B_{t+1}(1 - \beta_t) = B_t$ and $a_{t+1}\sigma_A = B_{t+1} \beta_t$, and summing the inequalities until $t=0$, we obtain:
\begin{align*}
B_{t}\|v_{t} - \theta^\star\|^2_{A^\dagger A} + 2A_{t} \Delta \hat{F}_A(x_{t}) \leq 2A_0 \Delta F_A(x_0) + B_0 \|v_0 - \theta^\star\|^2_{A^\dagger A}.
\end{align*}
We finish the proof by using the fact that $\psi(x_t) \leq \hat{\psi}_t$ and $\psi(x_0) = \hat{\psi}_0$ since $x_0 = v_0$.
\end{proof} 

Now that we have proven Theorem~\ref{thm:gen_apcg}, we can proceed to the proof of Lemma~\ref{lemma:lyapunov_psi}.

\begin{proof}[Proof of Lemma~\ref{lemma:lyapunov_psi}]
This lemma is a generalization of a part of the APCG to arbitrary probabilities (instead of uniform ones). It still uses the fact that $x_t$ can be written as a convex combination of $(v_l)_{l \leq t}$, but it requires to use a different convex combination for each coordinate of $x_t$, thus crucially exploiting separability of the proximal term. If coordinate $i$ is such that $\psi_i = 0$, then $\hat \psi^{(i)}_{t+1} \leq \alpha_t \psi_i(\tilde v_{t+1}^{(i)}) + (1 - \alpha_t) \hat \psi^{(i)}_{t}$ is automatically satisfied for any $\delta_t^{(i)}$. For coordinates $i$ such that $\psi_i \neq 0$ (and so $R_i = 1$), we start by expressing $x_{t+1}$ in terms of $x_t$, $v_{t+1}$ and $v_t$ . More precisely, we write that for any $t > 0$:
\begin{equation*}
     x_{t+1}^{(i)} = y_t^{(i)} + \frac{\alpha_t}{p_i} (v_{t+1}^{(i)} - w_t^{(i)}).
\end{equation*}
Indeed, either coordinate $i$ is updated at time $t$ or $v_{t+1}^{(i)} = w_t^{(i)}$ so the previous equation always holds. We can then develop the $w_t$ and $y_t$ terms to obtain $x_{t+1}^{(i)}$ only in function of $x_t^{(i)}$, $v_t^{(i)}$ and $v_{t+1}^{(i)}$:
\begin{align*}
    x_{t+1}^{(i)} 
    &= \frac{\alpha_t}{p_i}v_{t+1}^{(i)} + \left(1 - \frac{\alpha_t \beta_t}{p_i}\right) y_t^{(i)}  - \frac{\alpha_t(1 - \beta_t)}{p_i}v_{t}^{(i)}\\
    &= \frac{\alpha_t}{p_i}v_{t+1}^{(i)} + \left(1 - \frac{\alpha_t \beta_t }{p_i}\right)\frac{(1 - \alpha_t) x_t^{(i)} + \alpha_t(1 - \beta_t)v_t^{(i)}}{1 - \alpha_t \beta_t}  - \frac{\alpha_t(1 - \beta_t)}{p_i}v_{t}^{(i)}\\
    &= \frac{\alpha_t }{p_i}v_{t+1}^{(i)} + \alpha_t(1 - \beta_t)\left[\frac{1 - \frac{\alpha_t \beta_t }{p_i}}{1 - \alpha_t \beta_t} - \frac{1}{p_i} \right]v_t^{(i)} +  \left(1 - \frac{\alpha_t \beta_t }{p_i}\right)\frac{(1 - \alpha_t)}{1 - \alpha_t \beta_t} x_t^{(i)}\\
    &= \frac{\alpha_t }{p_i}v_{t+1}^{(i)} + \frac{\alpha_t(1 - \beta_t)}{1 - \alpha_t \beta_t}\left(1 - \frac{1}{p_i}\right)v_t^{(i)} + \left(1 - \frac{\alpha_t \beta_t }{p_i}\right)\frac{(1 - \alpha_t)}{1 - \alpha_t \beta_t} x_t^{(i)} .
\end{align*}
At this point, all coefficients sum to 1. Indeed, they all sum to 1 at the first line and we have expressed $w_t^{(i)}$ and then $y_t^{(i)}$ as convex combinations of other terms, thus keeping the value of the sum unchanged. Yet, $p_i < 1$ so the coefficient on the second term is negative. Fortunately, it is possible to show that the $v_t^{(i)}$ term in the decomposition of $x_t^{(i)}$ is large enough so that the $v_t^{(i)}$ term in the decomposition of $x_{t+1}^{(i)}$ is positive. More precisely, we now show by recursion that for $t \geq 0$:
\begin{equation}
\label{eq:xt_convex}
    x_{t+1}^{(i)} = \frac{\alpha_t }{p_i}v_{t+1}^{(i)} + \sum_{l=0}^{t} \delta^{(i)}_{t+1}(l) v_l^{(i)},
\end{equation}
with $\delta^{(i)}_{t+1}(l) \geq 0$ for $l \leq t$. For $t=0$, $x_0 = v_0$ and $x_1^{(i)} = \frac{\alpha_0 }{p_i}v_1^{(i)} + \left( 1 - \frac{\alpha_0 }{p_i}\right)v_0^{(i)}$. We now assume that Equation~\eqref{eq:xt_convex} holds for a given $t > 0$, and expand $\delta_{t+1}^{(i)}(t)$ to show that it is positive. Using that $\delta_{t}^{(i)}(t) = \frac{\alpha_t}{p_i}$, we write:
\begin{align*} 
    \delta_{t+1}^{(i)}(t) = &\frac{\alpha_t(1 - \beta_t)}{1 - \alpha_t \beta_t}\left(1 - \frac{1}{p_i}\right) + \frac{\alpha_t }{p_i} \left(1 - \frac{\alpha_t \beta_t }{p_i}\right)\frac{(1 - \alpha_t)}{1 - \alpha_t \beta_t} \\
    &= \frac{\alpha_t}{1 - \alpha_t \beta_t}\left[(1 - \beta_t) \left(1 - \frac{1}{p_i}\right) + \frac{(1 - \alpha_t) }{p_i}\left(1 - \frac{\alpha_t \beta_t }{p_i}\right)\right]\\
    &= \frac{\alpha_t}{1 - \alpha_t \beta_t}\left[1 - \beta_t - \frac{1}{p_i} + \frac{\beta_t }{p_i} + \frac{1}{p_i} - \frac{\alpha_t }{p_i} - (1 - \alpha_t)\frac{\alpha_t \beta_t}{p_i^2}\right]\\
    &= \frac{\alpha_t}{1 - \alpha_t \beta_t}\left[\left(1 - \beta_t - \frac{\alpha_t }{p_i}\right) + \frac{\beta_t }{p_i}\left(1 - (1 - \alpha_t)\frac{\alpha_t }{p_i}\right)\right].
\end{align*}
We conclude that $\delta_{t+1}^{(i)}(t) \geq 0$ since $1 - \beta_t - \frac{\alpha_t }{p_i} \geq 0$. Note that this condition can be weakened to $1 - \frac{\alpha_t^2 }{p_i^2} \geq 0$ when $\beta_t = \alpha_t$ or when $\beta_t = 0$. We also deduce from the form of $x_{t+1}^{(i)}$ that for $l < t$, the only coefficients on $v_l^{(i)}$ in the development of $x_{t+1}^{(i)}$ come from the $x_t^{(i)}$ term and so:
\begin{equation}
    \delta_{t+1}^{(i)}(l) = \left(1 - \frac{\alpha_t \beta_t }{p_i}\right)\frac{(1 - \alpha_t)}{1 - \alpha_t \beta_t} \delta_t^{(i)}(l),
\end{equation}
so these coefficients are positive as well. Since they also sum to $1$, it implies that $x_t^{(i)}$ is a convex combination of the $v_l^{(i)}$ for $l \leq t$, and we use the convexity of $\psi_i$ to write: $$\psi_i(x_t^{(i)}) = \psi_i\left(\sum_{l=0}^t \delta_t^{(i)}(l) v_l^{(i)}\right) \leq \sum_{l=0}^t \delta_t^{(i)}(l) \psi_i(v_l^{(i)}) =  \hat{\psi}^{(i)}_t.$$ 
Now, we can properly express $\hat{\psi}^{(i)}_{t+1}$ using the decomposition of $x_{t+1}^{(i)}$ in terms of $\delta_{t+1}^{(i)}$:
\begin{align*}
    \esp{\hat{\psi}^{(i)}_{t+1}} &= \esp{\frac{\alpha_t }{p_i}\psi_i(v_{t+1}^{(i)})} + \frac{\alpha_t(1 - \beta_t)}{1 - \alpha_t \beta_t}\left(1 - \frac{1}{p_i}\right)\psi_i(v_t^{(i)})\\
    &+ \left(1 - \frac{\alpha_t \beta_t }{p_i}\right)\frac{1 - \alpha_t}{1 - \alpha_t \beta_t}\sum_{l=0}^t \delta_t^{(i)}(l) \psi_i(v_l^{(i)})\\
    &= \alpha_t \psi_i(\tilde v_{t+1}^{(i)}) + \left(1 - p_i\right)\frac{\alpha_t}{p_i}\psi_i(w_{t}^{(i)})  + \frac{\alpha_t(1 - \beta_t)}{1 - \alpha_t \beta_t}\left(1 - \frac{1}{p_i}\right)\psi_i(v_t^{(i)})\\
    &+ \left(1 - \frac{\alpha_t \beta_t }{p_i}\right)\frac{1 - \alpha_t}{1 - \alpha_t \beta_t}\hat{\psi}^{(i)}_{t}\\
\end{align*}
At this point, we use the convexity of $\psi_i$ to develop $\psi_i(w_{t}^{(i)})$ and then $\psi_i(y_{t}^{(i)})$ in the following way:
\begin{align*}
     \psi_i(w_{t}^{(i)}) &\leq (1 - \beta_t)\psi_i(v_{t}^{(i)}) + \beta_t \psi_i(y_{t}^{(i)}) \\
     &\leq (1 - \beta_t)\psi_i(v_{t}^{(i)}) + \frac{\beta_t}{1 - \alpha_t \beta_t}\left[(1 - \alpha_t) \psi_i(x_{t}^{(i)}) + \alpha_t(1 - \beta_t) \psi_i(v_{t}^{(i)})\right]\\
     &= \frac{1 - \beta_t}{1 - \alpha_t \beta_t}\psi_i(v_{t}^{(i)}) + \frac{\beta_t(1 - \alpha_t)}{1 - \alpha_t \beta_t} \psi_i(x_{t}^{(i)}).
\end{align*}

If we plug these expressions into the development of $\esp{\hat{\psi}^{(i)}_{t+1}}$, the $\psi_i(v_{t}^{(i)})$ terms cancel and we obtain:

\begin{align*}
    \esp{\hat{\psi}^{(i)}_{t+1}} &\leq \alpha_t \psi_i(\tilde v_{t+1}^{(i)}) + \alpha_t\left(\frac{1}{p_i} - 1\right)\frac{\beta_t(1 - \alpha_t)}{1 - \alpha_t \beta_t} \psi_i(x_{t}^{(i)}) + \left(1 - \frac{\alpha_t \beta_t }{p_i}\right)\frac{1 - \alpha_t}{1 - \alpha_t \beta_t}\hat{\psi}^{(i)}_{t}\\
\end{align*}

We now use the fact that $\psi_i(x_{t}^{(i)}) \leq \hat{\psi}^{(i)}_{t}$ (by convexity of $\psi_i$) to get:

\begin{align*}
    \esp{\hat{\psi}^{(i)}_{t+1}} &\leq \alpha_t \psi_i(\tilde v_{t+1}^{(i)}) + \frac{1 - \alpha_t}{1 - \alpha_t \beta_t}\left[\alpha_t \beta_t \left(\frac{1}{p_i} - 1\right) + \left(1 - \frac{\alpha_t \beta_t }{p_i}\right)\right]\hat{\psi}^{(i)}_{t}\\
    &\leq \alpha_t \psi_i(\tilde v_{t+1}^{(i)}) + (1 - \alpha_t)\hat{\psi}^{(i)}_{t}
\end{align*}

This holds for any coordinate $i$ and so $\esp{\hat{\psi}_{t+1}} \leq \alpha_t \psi(\tilde v_{t+1} + (1 - \alpha_t)\hat{\psi}_{t}$ for all $t \geq 0$, which finishes the proof of the lemma.
\end{proof}

\subsection{Proof of the corollaries}
Now that that we have proven the main result, we show how specific choices of parameters lead to fast algorithms. 

\begin{proof}[Proof of Corollary~\ref{corr:sc_apcg}]
If $\sigma_A > 0$, then the parameters can be chosen as $\alpha_t = \beta_t = \rho = \frac{\sqrt{\sigma_A}}{S}$, with $A_t = (1 - \rho)^{-t}$ and $B_t = \sigma_A A_t$. These expressions can then be plugged into the recursion to verify that they do satisfy it. This choice of keeping a constant $\alpha_t$ is classic and slightly suboptimal for small values of $t$ compared with the choice made~\citet{nesterov2017efficiency}.
\end{proof}

\begin{proof}[Proof of Corollary~\ref{corr:cvx_apcg}]
We first prove that $\alpha_t$ can actually be obtained by a simple recursion. This comes from the (well-known) fact that the recursions in~\citet{lin2015accelerated} and~\citet{nesterov2017efficiency} are actually the same. If $\sigma_A = 0$ then we have to choose $\beta_t = 0$ for all $t$. Then, we can choose $B_t = B_0$ for any $B_0 > 0$. This allows to write $(A_{t+1} - A_t)^2S^2 = A_tB_0$ for all $t$, which is a second degree polynomial in the variable $A_{t+1}$. We choose the positive root in order to have $a_{t+1} \geq 0$, which yields: $$A_{t+1} = A_t + \frac{B_0}{2S^2}\left(1 + \sqrt{1 + 4S^2 B_0^{-1} A_t}\right).$$
Coefficients $(a_t)$ can be computed using
$$a_{t+1} = A_{t+1} - A_t = \frac{B_0}{2S^2}\left(1 + \sqrt{1 + 4S^2 B_0^{-1} A_t}\right),$$
and so we use the fact that $a_{t+1}S^2 = A_{t+1}B_{t+1}$, which can be rewritten as $\alpha_t = \frac{B_0}{a_{t+1}S^2}$. to obtain the sequence $(\alpha_t)$ as:
$$\alpha_t = \frac{2}{1 + \sqrt{1 + 4S^2B_0^{-1}A_t}}.$$ 
In particular, $$A_t = \left[\left(\frac{2}{\alpha_t} - 1\right)^2 - 1\right]\frac{B_0}{4S^2}.$$
This expression for $A_t$ and $A_{t+1}$ can be substituted in the relation $A_{t+1} = A_t + \frac{B_0}{a_{t+1} S^2}$, which yields after some simplifications: $$\alpha_{t+1}^{-2} - \alpha_{t+1}^{-1} - \alpha_t^{-2} = 0,$$ which is a second degree polynomial in the variable $\alpha_{t+1}^{-1}$. Solving for $\alpha_t$ leads to
$$\alpha_{t+1} = \frac{2}{1 + \sqrt{1 + 4\alpha_t^{-2}}} = \frac{\sqrt{\alpha_t^4 + 4 \alpha_t^2} - \alpha_t^2}{2},$$
which is the exact same recursion as in~\citet{lin2015accelerated} and~\citet{fercoq2015accelerated}. In particular, only the value of $\alpha_0$ matters and only the sequence $\alpha_t$ actually needs to be computed, since the only coefficients needed are the $\alpha_t$ and $\frac{a_{t+1}}{B_{t+1}} = \frac{1}{\alpha_t S^2}$.

We would like to choose the highest possible $\alpha_0$, such that $1 - \alpha_0 / p_{\min} \geq 0$, so we take $\alpha_0 = p_{\min}$ where $p_{\min} = \min_i p_i$ where the minimum is over all coordinates such that $\psi_i \neq 0$. This is enough to respect the condition $\alpha_t \leq p_{\min}$ since $(\alpha_t)$ is a decreasing sequence. This leads to $$A_0 = \left[\left(\frac{2}{p_{\rm min}} - 1\right)^2 - 1\right]\frac{B_0}{4S^2} \leq \frac{B_0}{p_{\rm min}^2 S^2}.$$ Since $A_0 \geq 0$, a direct recursion yields $A_t \geq \frac{B_0 t^2}{4S^2}$. We call $r_t^2 = \|v_0 - \theta^\star_A\|^2_{A^\dagger A} - \mathbb{E}[\|v_t - \theta^\star_A\|^2_{A^\dagger A}]$, and $\Delta F_t =  \mathbb{E}[q_A(x_t) + \psi(x_t)] - q_A(\theta^\star_A) + \psi_A(\theta^\star_A)$, then:
$$F_t \leq \frac{1}{2A_t}\left(B_0r_t^2 + 2A_0F_0\right) = \frac{B_0}{2A_t}\left(r_t^2 + \frac{2}{S^2 p_{\min}^2}F_0\right) \leq \frac{2S^2}{t^2}\left(r_t^2 + \frac{2}{S^2 p_{\min}^2}F_0\right),$$
which finishes the proof of the rate.

\end{proof}

\section{Algorithm Performances}
\label{app:algo_perfs}
The linear convergence rate of \adfs~is a direct consequence of the generalized APCG convergence theorem. Yet, it is not straightforward to derive hyperparameters that lead to a rate that is fast and that can be easily interpreted. The goal of this section is to choose such parameters when the functions $f_{i,j}$ are smooth, and detail the rate in this case. 

\subsection{Strong convexity of the augmented problem}

The number of iterations required to solve the augmented problem only depends on the conditioning of the augmented problem. The Hessian of $q_A$ is equal to $A^\top  \Sigma^{-1} A$ and the rate of Accelerated Proximal Coordinate Gradient depends on $\lambda_{\max}((A^\dagger A)U_b P^{-1} A^\top  \Sigma^{-1} A P^{-1} U_b (A^\dagger A))$. We study in this section the smallest eigenvalue of $A^\top  \Sigma^{-1} A$, and in particular prove Lemma~\ref{lemma:main_sigma_A_lb}. 
We start by proving a first lemma.

\begin{lemma}\label{lemma:ker_u_v}
Let $U, V \in \R^d$ be two symmetric positive semi-definite matrices. Let $x \in \Ker(U - V)$, that can be decomposed into $x= x_+ + x_\perp$, with $x_\perp \in \Ker(U)$ and $x_+ \in \Ker(U)^\perp$. Then, $x_+^\top U x_+ \leq x_+^\top V x_+$, and if $x_+ = 0$ then $x_\perp \in \Ker(V)$.
\end{lemma}

\begin{proof}
Let $x \in \Ker(U-V)$. We write:
$$x_+^\top U x_+ = x^\top U x =  x^\top V x = x_+^\top V x_+ + 2 x_\perp^\top V x_+ + x_\perp^\top V x_\perp.$$
Besides, $x_\perp^\top (U - V) x = 0$, and so $x_\perp^\top V x_+ = - x_\perp^\top V x_\perp \leq 0$. Therefore, 
$$x_+^\top U x_+ = x_+^\top V x_+ +  x_\perp^\top V x_+ \leq x_+^\top V x_+.$$
Finally, if $x_+ = 0$ then $x_\perp^\top V x_\perp = - x_\perp^\top V x_+ = 0$, and so $x_\perp \in \Ker(V)$.
\end{proof}

\begin{proof}[Proof of Lemma~\ref{lemma:main_sigma_A_lb}]
For any rectangular matrix $Q$, all non-zero singular values of the matrix $Q^\top Q$ are also non-zero singular values of the matrix $QQ^\top $, so we can analyze the spectrum of the matrix $\tilde{L} = \Sigma^{-1/2} A A^\top  \Sigma^{-1/2}$ instead of the spectrum of $A^\top \Sigma^{-1} A$. Recall that $A$ writes:
\begin{equation}
    A = \begin{pmatrix}
    A_\comm \otimes I_d & D_\mu \\
    0 & - D_\mu^{\rm diag}
    \end{pmatrix}
\end{equation}
Then, if we denote $L_\comm$ the Laplacian matrix of the original true graph, the rescaled Laplacian matrix of the augmented graph writes:
\begin{equation}
\tilde{L} = \Sigma^{-1/2}\begin{pmatrix} L_\comm \otimes I_d + D_\mu D_\mu^\top & - D^\mu D_\mu^{\rm diag} \\
- D_\mu^{\rm diag} D_\mu^\top  & (D_\mu^{\rm diag})^2 \end{pmatrix}\Sigma^{-1/2}.
\end{equation}
We define 
\begin{align*}
    P_{\perp} =
\begin{pmatrix}
     P_{11} & & 0 \\
     & \cdots &  \\
     0 & & P_{nm}
\end{pmatrix} \in \R^{nmd \times nmd}.
\end{align*}
If we split $\Sigma$ into two diagonal blocks $\Sigma_\comm = {\rm diag}(\sigma_1, \cdots, \sigma_n) \otimes I_d$ (for the communication nodes), and $\Sigma_\comp$ (for the computation nodes) and apply the block determinant formula, we obtain:
\begin{align*}
    \det&(\tilde{L} - \lambda I_{n(m+1)d}) = \det(\Sigma_\comp^{-\frac{1}{2}} (D_\mu^{\rm diag})^2 \Sigma_\comp^{-\frac{1}{2}} - \lambda I_{nmd})\\
    &\times \det(\Sigma_\comm^{-\frac{1}{2}}[ (L_\comm \otimes I_d) + D_\mu D_\mu^\top - 
    \lambda \Sigma_\comm -\\
    &D_\mu D_\mu^{\rm diag}  \Sigma_\comp^{-\frac{1}{2}} \left(\Sigma_\comp^{-\frac{1}{2}} (D_\mu^{\rm diag})^2 \Sigma_\comp^{-\frac{1}{2}}- \lambda I_{nmd} \right)^\dagger \Sigma_\comp^{-\frac{1}{2}} D_\mu^{\rm diag} D_\mu^\top] \Sigma_\comm^{-\frac{1}{2}}).
\end{align*}
Now, we use that $\mu_{ij}^2 = \alpha L_{ij}$ for some $\alpha > 0$, and note that:
$$ (\Sigma_\comp^{-\frac{1}{2}} (D_\mu^{\rm diag})^2 \Sigma_\comp^{-\frac{1}{2}})_{ij} = (D_\mu^{\rm diag} \Sigma_{\comp}^{-1} D_\mu^{\rm diag})_{ij} = \mu_{ij}^2 P_{ij} / L_{ij} = \alpha P_{ij}.$$
This can be used along with the fact that $D_\mu D_\mu^\top = D_\mu P_\perp D_\mu^\top$ to rewrite the previous determinant as:
\begin{align*}
    \det&(\tilde{L} - \lambda I_{n(m+1)d}) = \det(\alpha P_\perp - \lambda I_{nmd}) \times \\
    &\det(L_\comm \otimes I_d - \lambda \Sigma_\comm - \alpha D_\mu P_\perp (\alpha P_\perp - \lambda I_{nmd})^\dagger D_\mu^\top).
\end{align*}
Note that since $P_\perp$ is a projector,
$$P_\perp (\alpha P_\perp - \lambda I_{nmd})^\dagger =  ((\alpha P_\perp - \lambda I_{nmd})P_\perp)^\dagger = (\alpha - \lambda)^{-1} P_\perp.$$
Therefore, the non-zero eigenvalues of $A^\top \Sigma^{-1} A$ are the $\lambda$ that satisfy the following equation:
\begin{equation}\label{eq:lambda_det_equation}
    0 =\det(\alpha P_\perp - \lambda I_{nmd}) \det\left( L_\comm \otimes I_d - \lambda \left(\Sigma_\comm + \frac{1}{\alpha - \lambda}D_\mu D_\mu^\top\right)\right).
\end{equation}
We now consider $0 < \lambda \leq \alpha / 2$, so that:
$\Sigma_\comm + \frac{1}{\alpha - \lambda}D_\mu D_\mu^\top \preccurlyeq \tilde{D}_M \otimes I_d$, with 
$$(\tilde{D}_M)_{ii} = \lambda_{\max}\left(\sigma_i I_d + \frac{2}{\alpha}\sum_{j=1}^m \mu_{ij}^2 P_{ij}\right) = \sigma_i + 2 \lambda_{\max}\left(\sum_{j=1}^m L_{ij} P_{ij}\right).$$
Therefore, for any $y \in \Ker(L_\comm \otimes I_d)^\perp$ such that $y \neq 0$, 
$$y^\top(L_\comm \otimes I_d - \Delta_\lambda) y > y^\top ( (L_\comm - \lambda \tilde{D}_M) \otimes I_d) y.$$
In particular we have that if $0 < \lambda < \lambda_{\min}^+(\tilde{D}_M^{-\frac{1}{2}} L_\comm \tilde{D}_M^{-\frac{1}{2}})$ then
\begin{equation}\label{eq:lambda_min_tilde}
    y^\top(L_\comm \otimes I_d - \Delta_\lambda) y > 0.
\end{equation}
Let $x \in {\rm Ker}(L_\comm \otimes I_d - \Delta_\lambda)$ then Lemma~\ref{lemma:ker_u_v} tells us that $x = x_+ + x_\perp$ with $x^+ \in \Ker(L_\comm \otimes I_d)^\perp$ and $x_+^\top (L_\comm \otimes I_d - \Delta_\lambda) x_+ \leq 0$. If $x_+ \neq 0$ then this contradicts Equation~\eqref{eq:lambda_min_tilde} so $x_+ = 0$, meaning that $x \in \Ker(\Delta_\lambda)$ using the second part of Lemma~\ref{lemma:ker_u_v}, and so $x = 0$.

Therefore, if $\lambda$ is such that $0 < \lambda < \min(\lambda_{\min}^+(\tilde{D}_M^{-\frac{1}{2}} L_\comm \tilde{D}_M^{-\frac{1}{2}}), \alpha / 2)$ then by using the eigenvalue characterization given by Equation~\eqref{eq:lambda_det_equation}, $\lambda$ is not an eigenvalue of $A^\top \Sigma^{-1} A$ since ${\rm Ker}(L_\comm - \Delta_\lambda) = \{0\}$, so in particular $\lambda_{\min}^+(A^\top \Sigma^{-1} A) \geq \min(\tilde{D}_M^{-\frac{1}{2}} L_\comm \tilde{D}_M^{-\frac{1}{2}}, \alpha / 2)$.

Besides, $\lambda_{\min}^+(\tilde{D}_M^{-\frac{1}{2}} L_\comm \tilde{D}_M^{-\frac{1}{2}}) = \lambda_{\min}^+(A^\top_\comm \tilde{D}_M^{-1} A_\comm)$ and $A^\top_\comm \tilde{D}_M^{-1} A_\comm$ is independent of $\alpha$ since $\alpha$ only affects the $\mu_{ij}$ weights when $(i,j)$ is a computation edge, and so we can choose $\alpha = 2\lambda_{\min}^+(A_\comm^\top \tilde{D}_M^{-1} A_\comm)$, so that $\lambda_{\min}^+(A^\top \Sigma^{-1} A) \geq \lambda_{\min}^+(A^\top_\comm \tilde{D}_M^{-1} A_\comm)$.
Finally, $\Ker(A^\top  \Sigma^{-1} A) = \Ker(A)$ and so 
\begin{equation*}
    \|x\|^2_{A^\top \Sigma^{-1} A} \geq \lambda_{\min}^+(A^\top \Sigma^{-1} A) \|x\|^2_{A^\dagger A} \geq  \lambda_{\min}^+(A_\comm^\top \tilde{D}_M^{-1} A_\comm) \|x\|^2_{A^\dagger A},
\end{equation*}
which finishes the proof. 
\end{proof}

\subsection{Smoothness of the augmented problem}
The goal of this section is to prove Lemma~\ref{lemma:rate_adfs} by analyzing $\lambda_{\max}((A^\dagger A)^\top  P_b^\dagger M P_b^\dagger A^\dagger A)$ for any block $b$.

\begin{proof}[Proof of Lemma~\ref{lemma:rate_adfs}]
The proof is split into three parts. We first bound the value of $\lambda_{\max}((A^\dagger A)^\top  P_b^\dagger A^\top \Sigma^{-1}A P_b^\dagger A^\dagger A)$ depending on whether $b$ is a communication or a computation block, and then we give a bound on the rate $\rho$.\\

\textbf{Communication blocks.} Under the sampling of Assumption~\ref{assumption:synchronous_sampling}, all coordinates have the same probability $p_e$ of being selected at each step. In this case, $P_b^\dagger = \frac{1}{p_b} U_b$ where $U_b$ is the projector on communication edges that are in $b$ (all the communication edges for Assumption~\ref{assumption:synchronous_sampling}. We denote $V_b \in \R^{(m +1)nd \times (m+1)nd}$ the projector on $\{i, \ \exists j / (i,j) \in b\}$, the set of nodes for which one of their vertices is updated, and write:
\begin{align*}
    P_b^\dagger A^\top  \Sigma^{-1} A P_b^\dagger &=  \frac{1}{p_e^2}U_b A^\top  \Sigma^{-1} A U_b = \frac{1}{p_e^2}U_b A^\top  V_b \Sigma^{-1} V_b A U_b
\end{align*}
In this case, we note $L_b$ the Laplacian of the subgraph defined by the edges in $b$, which is such that $L_b \otimes I_d = A U_b A^\top$, and we use the fact that $\lambda_{\max}(A U_b A^\top ) = \lambda_{\max}(U_b A^\top  A U_b)$ to write:
\begin{equation}
\label{eq:comm_smoothness_general}
    \lambda_{\max}((A^\dagger A)^\top  P_b^\dagger A^\top \Sigma^{-1}A P_b^\dagger A^\dagger A) \leq \frac{\lambda_{\max}(L_b)}{\sigma_{\min} p_e^2}.
\end{equation}

In particular, Equation~\eqref{eq:comm_smoothness_general} allows to consider dynamically changing graphs for which we know that all edges have the same probability of appearing at each step and for which we can bound the Laplacian matrix of any subgraph. This allows to consider a complete underlying communication graph while taking advantage of communications on subgraphs only. If we take $p_e = \pcomm$ (i.e. all communication edges are sampled at each communication step) then this becomes:
$$\lambda_{\max}((A^\dagger A)^\top  P_b^\dagger A^\top \Sigma^{-1}A P_b^\dagger A^\dagger A) \leq \frac{\lambda_{\max}(L)}{\sigma_{\min} \pcomm^2},$$
with $L$ the Laplacian matrix of the original communication graph.\\

\textbf{Computation blocks.} We start with the case in which each node only samples the coordinate associated with one virtual edge. In this case, we take $\lambda \in \R^{E + nmd}$ and write:
\begin{align*}
  &(e_b \otimes \lambda)^\top P_b^\dagger A^\top  \Sigma^{-1} A P_b^\dagger (e_b \otimes \lambda)\\
  &= \sum_{i=1}^n \sum_{j, (i,j) \in b} e_{ij}^\top  A^\top  \sum_{u \in V} \Sigma_{uu}^{-1}  e_u e_u^\top  \sum_{i^\prime=1}^n \sum_{j^\prime, (i^\prime,j^\prime) \in b} A e_{i^\prime j^\prime} \times \lambda_{ij}^\top P_{ij} P_{i^\prime j^\prime} \lambda_{i^\prime j^\prime}\\
  &= \sum_{i=1}^n \sum_{j, (i,j) \in b} \sum_{i^\prime=1}^n \sum_{j^\prime, (i^\prime,j^\prime) \in b} \frac{\mu_{ij}}{p_{ij}} \sum_{u \in V} \Sigma_{uu}^{-1} (e_i - e_j)^\top  e_u e_u^\top  (e_i^\prime - e_j^\prime) \lambda_{ij}^\top P_{ij} P_{i^\prime j^\prime} \lambda_{i^\prime j^\prime}  \\
  &= \sum_{i=1}^n \sum_{j, (i,j) \in b} \frac{\mu_{ij}^2(\sigma_i^{-1} + L_{i,j}^{-1})}{p_{ij}^2}  \times  \|\lambda_{ij}\|^2_{P_{ij}}.
\end{align*}
We deduce that if only one coordinate is sampled per node then we have:
$$\lambda_{\max}(P_b^\dagger A^\top  \Sigma^{-1} A P_b^\dagger) \leq \max_{i,j} \frac{\mu_{ij}^2(\sigma_i^{-1} + L_{i,j}^{-1})}{p_{ij}^2}.$$\\

\textbf{Rate of convergence.}
Recall that the rate of convergence of Synch-ADFS can be written as:
$$\rho^2 = \min_b \frac{\sigma_A}{\lambda_{\max}((A^\dagger A)^\top  P_b^\dagger M P_b^\dagger A^\dagger A)} = \min_b \frac{\lambda_{\min}^+(A_\comm^\top \tilde{D}_M^{-1} A_\comm)}{\lambda_{\max}((A^\dagger A)^\top  P_b^\dagger A^\top \Sigma^{-1} A P_b^\dagger A^\dagger A)},$$
where $(\tilde{D}_M)_{ii} = \sigma_i + 2\lambda_{\max}\left(\sum_{j=1}^m L_{i,j} P_{ij}\right)$. If we take $b$ to be the set of all communication edges, then we obtain: 
\begin{align*}
&\rho_{\rm comm}^2 = \frac{\lambda_{\min}^+(A_\comm^\top \tilde{D}_M^{-1} A_\comm)}{\lambda_{\max}(A_\comm^\top \Sigma_\comm^{-1} A_\comm)}\pcomm^2 = \frac{\gamma \pcomm^2}{\kappa_\comm}, \\
&\rho_{\rm comp}^2 = \min_{ij} \frac{p_{ij}^2 \sigma_A}{\mu_{ij}^2(\sigma_i^{-1} + L_{ij}^{-1})} = \min_{ij} \frac{p_{ij}^2}{2(1 + L_{ij}\sigma_i^{-1})},
\end{align*}
where we used in the second equation that $\mu_{ij}^2 = \alpha L_{ij}$ and $\alpha = 2 \sigma_A$. The constraint on $\rho_{\rm comp}^2$ is that all $p_{ij}$ are normalized separately, i.e. $\sum_{j} p_{ij} = 1$ for each node $i$. Indeed, exactly one sample per node is chosen at each step, and so:
\begin{equation} \label{eq:rho_comp_lb}
    \rho_{\rm comp}^2 = \frac{\pcomp^2}{2 S_i^2} \geq \frac{\pcomp^2}{2 S_{\max}^2}.
\end{equation}
We finally use the concavity of the square root with Jensen inequality to get: $$S_i = \sum_{j=1}^m\sqrt{1 + L_{ij}\sigma_i^{-1}} \leq m\sqrt{\sum_{j=1}^m\frac{1}{m}(1 + L_{ij}\sigma_i^{-1})} =  m\sqrt{1 + \sum_{j=1}^m\frac{L_{ij}}{m\sigma_i}},$$
which yields $S_{\max}^2 \leq m^2 + m\kappa_s$, or $S_{\max} \leq m + \sqrt{m\kappa_s}$.
\end{proof}

\subsection{Execution time}

\begin{proof}[Proof of Theorem~\ref{thm:adfs_speed_precise}]
The execution time of the algorithm $T(K)$ verifies the following bound:
\begin{equation}
    \mathbb{E}[T(K)] = \left(p_{\rm comp} +  \tau p_{\rm comm}\right) K
\end{equation}
Algorithm~\ref{algo:sc_adfs} requires $ - \log(1 / \varepsilon) / \log(1 - \rho)$ iterations to reach error $\varepsilon$. Using that $\log(1 + x) \leq x$ for any $x > -1$, we get that using $K_\varepsilon = \log(1 / \varepsilon) \rho^{-1}$ instead also guarantees to make error less than $\varepsilon$. We now optimize the bound in $\rho$:
\begin{equation}
    \frac{\esp{T(K_\varepsilon)}}{\log\left(\varepsilon^{-1}\right)} = \rho^{-1} \left(p_{\rm comp} + \tau p_{\rm comm}\right)
\end{equation}
If we rewrite this in terms of $\rho_{\rm comm}$ and $\rho_{\rm comp}$, we obtain:
\begin{equation}
     \frac{\esp{T(K_\varepsilon)}}{\log\left(\varepsilon^{-1}\right)} = \max \left(T_1(p_{\rm comm}), T_2(p_{\rm comm})\right)
\end{equation}
\begin{align}
\label{eq:T1_pcomm}
    &T_1(p_{\rm comm}) = \rho_{\rm comm}^{-1} (p_{\rm comp} + \tau p_{\rm comm}) = C_\comm \left(\tau - 1 + \frac{1}{p_{\rm comm}}\right),\\
    &T_2(p_{\rm comm}) = \rho_{\rm comp}^{-1} (\pcomp + \tau p_{\rm comm}) = C_\comp \left(1 + \tau \frac{\pcomm}{1 - \pcomm}\right).
\end{align}
with $C_\comm^2 = \frac{\kappa_\comm}{\gamma}$ and $C_\comp^2 = 2S^2_{\max}$ which are independent of $p_\comp$ and $p_\comm$.
$T_1$ is a continuous decreasing function of $p_{\rm comm}$ with $T_1 \rightarrow \infty$ when $p_{\rm comm} \rightarrow 0$. Similarly, $T_2$ is a continuous increasing function of $p_{\rm comm}$ such that $T_2 \rightarrow \infty$ when $p_{\rm comm} \rightarrow 1$. Therefore, the best upper bound on the execution time is given by taking $p_{\rm comm} = p^*$ where $p^*$ is such that $T_1(p^*) = T_2(p^*)$ and so $\rho_{\rm comm}(p^*) = \rho_{\rm comp}(p^*)$.
\begin{equation}
    \frac{\esp{T(K_\varepsilon)}}{\log\left(\varepsilon^{-1}\right)} = T_1(p^*)
\end{equation}
Then, $p^*$ can be found by finding the root in $]0, 1[$ of a second degree polynomial. In particular, $p^*$ is the solution of:
\begin{equation}
    p_{\rm comp}^2 = p_\comm^2 \frac{C_\comp}{C_\comm} = (1 - p_{\rm comm})^2
\end{equation}
which leads to
$$p^* = \left( 1 + \frac{C_\comp}{C_\comm}\right)^{-1}.$$
Plugging it back into Equation~\eqref{eq:T1_pcomm}, we get:
\begin{equation}
    \frac{\esp{T(K_\varepsilon)}}{\log\left(\varepsilon^{-1}\right)} = C_\comp + C_\comm \tau ,
\end{equation}
and so:
$$\frac{\esp{T(K_\varepsilon)}}{\log\left(\varepsilon^{-1}\right)} = \sqrt{2}(m + \sqrt{m\kappa_s}) + \tau \sqrt{\frac{\kappa_\comm}{\gamma}},$$
which completes our proof.
\end{proof}

\subsection{Optimality of ADFS}
\label{app:adfs_line_graph}
We know that in the homogeneous setting it is possible to recover $\kappa_\comm = \kappa_b$ and so ADFS is optimal. Yet, the worst-case function used in the proof of Theorem~\ref{thm:lb_general} is such that
\begin{align*}
    &\Sigma_\comm \succcurlyeq \frac{\sigma}{3n}D_n \hbox{ and } \tilde{D}_M \preccurlyeq \frac{L}{n}D_n, \hbox{ with } D_n = {\rm Diag}(n, 2, \cdots, 2) \in \R^n.
\end{align*}
Note that this control on $\tilde{D}_M$ is quite loose for the nodes that are not at the end of the line, but actually yields rather tight results. For a line graph, $A_\comm^\top D_n^{-1} A_\comm = (A_\comm^\top A_\comm - \mu_\comm^2(1 - 2n^{-1})e_{12}e_{12}^\top) / 2$. Therefore, 
\begin{equation}\label{eq:lambda_max_line}
    \lambda_{\max}(A_\comm^\top \Sigma_\comm^{-1} A_\comm) \leq \frac{3n\lambda_{\max}(A_\comm^\top A_\comm)}{2\sigma}.    
\end{equation}
For the second part, we note $P_{n, \alpha}$ the characteristic polynomial of the matrix $A_\comm^\top {\rm Diag}(\alpha, 1, \cdots, 1)A_\comm$, which is such that 
\begin{equation} \label{eq:recursion_P_lambda}
    P_{n+1, \alpha}(\lambda) = (1 + \alpha - \lambda) P_{n, 1}(\lambda) - P_{n-1, 1}(\lambda).    
\end{equation}
Unrolling the recursion for $\alpha = 1$, one can verify that $P_{n,1}(\lambda)$ is of the form:
\begin{equation}
    P_{n, 1}(\lambda) =  \frac{\sin(n\theta)}{\sin \theta}, \hbox{ with } \cos(\theta) = 1 - \frac{\lambda}{2},
\end{equation}
where we recall that $n$ is the number of nodes of the graph. Therefore, $P_{n,1}(\lambda)$ has a simple expression, and its roots recover the standard eigenvalues for the line graph, which are $\lambda_k = 2(1 - \cos(k \pi/n))$ for $0 \leq k < n$. We are interested in the roots of $P_{n, \alpha}$, which we express as:
\begin{align*}
    P_{n, \alpha}(\lambda) &= P_{n, 1}(\lambda) - (1 - \alpha) P_{n-1, 1}(\lambda)\\
    &= \frac{\sin(n\theta) - (1 - \alpha) \sin((n-1)\theta)}{\sin(\theta)}.
\end{align*}
Recall that we consider $0 \leq \alpha \leq1$, so if $\theta$ is a root of $P_{n, \alpha}$ then $\sin(n\theta) \leq \sin((n-1)\theta)$. Yet, the $\sin$ function is increasing on $[0, \pi /2]$, and thus we deduce that the roots of $P_{n, \alpha}$ are such that $\theta \geq \frac{\pi}{2n}$. Therefore, $\theta_{\min}^+(\alpha) \geq \theta_{\min}^+(1) / 2$, which implies that:
\begin{equation}\label{eq:lambda_min_line}
    \lambda_{\min}^+(A_\comm^\top \tilde{D}_M^{-1} A_\comm) = O\left(n \frac{\lambda_{\min}^+(A_\comm^\top A_\comm)}{L}\right).
\end{equation}
In the end, we combine Equations~\eqref{eq:lambda_max_line} and~\eqref{eq:lambda_min_line} and obtain that $\kappa_\comm = O\left(\kappa_b\right)$, and so the lower bound is matched up to constants in this setting. Note that we obtain $\kappa_b$ and not $\kappa_s$ because 
$$\max_{i \in \{1, \cdots, n\}} \lambda_{\max}\left(\sum_{j=1}^m L_{ij} P_{ij}\right) = \max_{i \in \{1, \cdots, n\}, j \in \{1, \cdots, m\}} L_{ij}$$
in this case.
\section{Efficient versions of ADFS}
\label{app:adfs_efficient}
We present in this section the efficient versions of Algorithm~\ref{algo:sc_adfs} and Algorithm~\ref{algo:ns_adfs}. These versions get rid of the convex combinations that can be very costly in a high dimensional setting. Instead, all the local updates can take advantage of the sparsity of the updates. Yet, full-dimensional updates are still required for communications. 

At first glance, a full dimensional operation seems to be required to perform the gradient step of Line 4. Yet, in Algorithm~\ref{algo:generalized_apcg_eff_sc}, $g_t$ is only used inside of the proximal term of Line $10$. Since $f_{ij}^*(x) = + \infty$ if $x \notin {\rm Ker}(f_{i,j})^\perp$ then only the projection of $g_t^{(i,j)}$ onto ${\rm Ker}(f_{i,j})^\perp$ actually matter. In particular, only $X_{i,j}^\top  g_t$ matters if $f_{i,j}(x) = \ell(X_{i,j}^\top x)$. Therefore, computations can dramatically be reduced if $X_{i,j}$ is sparse.

\begin{algorithm}
\caption{ADFS-EFFICIENT$\left(A, (\sigma_i), (L_{i,j}), (\mu_{k\ell}), (p_{k\ell}), \rho \right)$}
\label{algo:sc_adfs_eff}
\begin{algorithmic}[1]
\STATE $\sigma_A = \lambda_{\min}^+(A^\top \Sigma^{-1}A)$, $\tilde{\eta}_{k\ell} = \frac{\rho}{\sigma_A}$, $W_b = AP_b^\dagger A^\top $, $\tilde{W}_b = AP_b^\dagger A^\dagger$ $\phi = $\COMMENT{Initialization}
\STATE $x_0 = y_0 = z_0 = 0^{(n + nm) \times d}$ 
\FOR[Run for $K$ iterations]{$t=0$ to $K-1$}
\STATE Sample set of edges $b$ \COMMENT{Edges sampled from the augmented graph}
\STATE $g_t = \eta W_b \Sigma^{-1}(\phi^{t+1} u_t + z_t)$ \COMMENT{Communication abstracted by the matrix $W_b$.}
\IF{$b$ is a set of virtual edges}
\FOR{$i = 1$ to $n$}
\FOR{$j$ such that $(i,j) \in b$}
\STATE $w_t^{(i,j)} = -\phi^{t+1}u_t^{(i,j)} + z_t^{(i,j)}$
\STATE $h_{t+1}^{(i,j)} = {\rm prox}_{\eta  \mu_{ij}^2 p_{ij}^{-1} \tilde{f^*_{ij}}}\left(w_t^{(i,j)} - g_t^{(i,j)}\right) - w_t^{(i,j)}$ 
\STATE $h_{t+1}^{(i)} = - \sum_{j, (i,j) \in b} h_{t+1}^{(i,j)}$ \COMMENT{Center node update}
\ENDFOR 
\ENDFOR
\ELSE
\STATE $h_t = - g_t$
\ENDIF

\STATE $u_{t+1} = u_t - \frac{I - \rho \tilde{W}_b}{2\phi^{t+1}}h_t, \ \ \ z_{t+1} = z_t + \frac{I + \rho \tilde{W}_b}{2}h_t$
\ENDFOR
\STATE \textbf{return} $\theta_K =  \Sigma^{-1} (\phi^{K+1}u_K + z_K)$ \COMMENT{Return primal parameter}
\end{algorithmic}
\end{algorithm}

\begin{algorithm}
\caption{NS-ADFS}
\label{algo:ns_adfs_eff}
\begin{algorithmic}[1]
\STATE $u_0 = v_0 = 0^{(n + nm) \times d}$, $t = 0$, $\alpha_0 = \min_{i, \psi_i \neq 0} p_i$, $\eta_t = \frac{1}{\alpha_t S^2}$, $W_b = AP_b^\dagger A^\top $, $\tilde{W}_b = AP_b^\dagger A^\dagger$ \COMMENT{Initialization} 
\FOR[Run for $K$ iterations]{$t=0$ to $K-1$}
\STATE Sample set of edges $b$ \COMMENT{Edges sampled from the augmented graph}
\STATE $g_t = \eta_t W_b \Sigma^{-1}(\alpha_t^2 u_t + z_t)$ \COMMENT{Communication abstracted by the matrix $W_b$.}
\IF{$b$ is a set of virtual edges}
\FOR{$i = 1$ to $n$}
\FOR{$j$ such that $(i,j) \in b$}
\STATE $h_{t+1}^{(i,j)} = {\rm prox}_{\eta_t \mu_{ij}^2 p_{ij}^{-1} \tilde{f^*_{ij}}}\left(z_t^{(i,j)} - g_t^{(i,j)}\right) - z_t^{(i,j)}$ 
\STATE $h_{t+1}^{(i)} = - \sum_{j, (i,j) \in b} h_{t+1}^{(i,j)}$ \COMMENT{Center node update}
\ENDFOR 
\ENDFOR
\ELSE
\STATE $h_t = - g_t$
\ENDIF

\STATE $u_{t+1} = u_t - \frac{I - \alpha_t \tilde{W}_b}{2\alpha_t^2}h_t, \ \ \ z_{t+1} = z_t + h_t$
\STATE $\alpha_{t+1} = \frac{\sqrt{\alpha_t^4 + 4\alpha_t^2} - \alpha_t^2}{2}$
\ENDFOR
\STATE \textbf{return} $\theta_K =  \Sigma^{-1} (\alpha_t^2 u_K + z_K)$ \COMMENT{Return primal parameter}

\end{algorithmic}
\end{algorithm}

\section{Experimental setting}
\label{app:experimental_setting}

We detail in this section the exact experimental setting in which simulations were made. All algorithms used out-of-the-box parameters given by theory. Batch algorithms were given the exact $\kappa_b$. The datasets we used are the first million samples of the Higgs dataset (11 million samples and 28 attributes) and the Covtype.binary.scale dataset (581,012 samples and 54 attributes). Both datasets are available at \url{https://www.csie.ntu.edu.tw/~cjlin/libsvmtools/datasets/binary.html}. To obtain the local dataset $X_i \in \mathbb{R}^{m \times d}$ of each node, we drew $m$ samples at random from the base dataset, so that datasets of different nodes may overlap. We used the logistic loss with quadratic regularization, meaning that the function at node~$i$ is:
$$f_i: \theta_i \mapsto \sum_{j=1}^m \log\left(1 + \exp(-l_{i,j} X_{i,j}^\top \theta_i)\right) + \frac{\sigma_i}{2}\|\theta_i \|^2,$$
where $l_{i,j} \in \{-1, 1\}$ is the label associated with $X_{i,j}$, the $k$-th sample of node $i$. We chose $m = 10^4$ and $\sigma = 1$ for all simulations. Note that local functions are not normalized (not divided by $m$) so this actually corresponds to a regularization value of $\sigma_i = 10^{-4}$ with usual formulations. Computation delays were chosen constant equal to $1$ and communication delays constant equal to $5$.

Plots are shown for \emph{idealized times} in order to abstract implementation details as well as ensure that reported timings were not impacted by the cluster status (available bandwidth for example). This means that we counted $1$ unit of time for each computation step and $\tau$ for each communication step. The same setting as described in~\citet{hendrikx2019accelerated} was used for the locally synchronous version of ADFS, so nodes perform a schedule and are considered free to start the next iteration as soon as they send their a gradient as long as they already received the neighbor's gradient (non-blocking send). Note that although Algorithm~\ref{algo:sc_adfs} returns vector $\Sigma^{-1}v_t$ to compute the error, we used the vector $\Sigma^{-1}y_t$ instead. Both have similar asymptotic convergence rates but the error was more stable using $\Sigma^{-1}y_t$. The error that we plot is the average error over all nodes at a given time. More specifically, all nodes compute the error at specific iteration number as $F(\Sigma^{-1}y_t)$. Then, we average all these errors and the time reported is the time at which the last node finishes this iteration.

Similarly to Table~\ref{fig:table_speeds}, we assume that computing the dual gradient of a function $f_i$ is as long as computing $m$ proximal operators of $f_{i,j}$ functions. This greatly benefits to MSDA since in the case of logistic regression, the proximal operator for one sample has no analytic solution but can be efficiently computed by solving a one-dimensional optimization problem~\citep{shalev2013stochastic}, for example using Newton Method. The inner problem corresponding to computing $\nabla f_i^*$ was solved by performing $1000$ steps of accelerated gradient descent. For Point-SAGA, ADFS and DSBA, 1D prox were computed using 10 steps of Newton's method (in one dimension). Both used warm-starts, \emph{i.e.} the initial parameter for these inner problems was the solution for the last time the problem was solved. The step-size $\alpha$ of DSBA was chosen as $1 / (4L_{\max})$ instead of $1 / (24L_{\max})$ where $L_{\max} = \max_{i,j} L_{i,j}$ (it was unstable for larger $\alpha$).

\end{document}